\documentclass[10pt,reqno,oneside]{amsart}

\usepackage[all]{xy}
\usepackage{amssymb,latexsym}
\usepackage{mltex}
\usepackage{amsmath}
\usepackage[colorlinks=true,linkcolor=blue,citecolor=blue]{hyperref}
 \usepackage{lscape}
\usepackage{enumerate}
\usepackage{amsthm}
\usepackage{hyperref}
\usepackage{multicol}

\usepackage[totalwidth=14cm,totalheight=22cm]{geometry}

\newtheorem{lem}{Lemma}[section]
\newtheorem{thm}[lem]{Theorem}

\newtheorem{prop}[lem]{Proposition}

\theoremstyle{definition}

\theoremstyle{remark}
\newtheorem{rem}[lem]{Remark}

\newcommand{\R}{\mathbb{R}}

\DeclareRobustCommand{\gobblefive}[5]{}

\begin{document}
\title{The Dirichlet problem for the prescribed scalar curvature in Anti-de Sitter space}

\author{Pierre Bayard}

\address[P.~B.]{Facultad de Ciencias, Universidad Nacional Aut\'onoma de M\'exico
\\Av. Universidad 3000, Circuito Exterior S/N
\\Delegaci\'on Coyoac\'an, C.P. 04510, Ciudad Universitaria, CDMX, M\'exico}
\email{bayard@ciencias.unam.mx}

\maketitle

\begin{abstract}
We study the Dirichlet problem for the fully non-linear elliptic equation of second order traducing the prescription of the scalar curvature of a spacelike hypersurface in Anti-de Sitter space. The problem is solved if the datas are strictly convex.
\end{abstract}

\vspace{1cm}


\date{}
\maketitle\pagenumbering{arabic}


\section{Introduction}
Let us consider the space $\R^{n,2},$ which is $\R^{n+2}$ endowed with the metric
$$\sum_{i=1}^ndX_{i}^2-dX_{n+1}^2-dX_{n+2}^2$$
and the quadric model of Anti-de Sitter geometry
\begin{equation}\label{quadric model}
\mathbb{H}^{n,1}:=\{X=(X_1,\ldots,X_{n+2})\in\R^{n,2}:\ \sum_{i=1}^nX_i^2-X_{n+1}^2-X_{n+2}^2=-1\}.
\end{equation}
The metric induced on $\mathbb{H}^{n,1}$ has signature $(n,1)$ and constant sectional curvature $-1.$ $\mathbb{H}^{n,1}$ is not simply connected, and its universal covering space may be described as follows. Let us consider the totally geodesic hypersurface
\begin{equation}\label{intro hyperboloid model}
\mathbb{H}^{n}:=\mathbb{H}^{n,1}\cap\{X_{n+2}=0,\ X_{n+1}>0\}
\end{equation}
with its natural hyperbolic metric $g_H=\sum_{i=1}^ndX_{i}^2-dX_{n+1}^2,$ and the map 
\begin{eqnarray}
\mathbb{H}^n\times\R&\rightarrow&\mathbb{H}^{n,1}\label{intro map HnxR Q}\\
(X,t)&\mapsto& (X_1,\ldots,X_n,\cos t\ X_{n+1},\sin t\ X_{n+1}).\nonumber
\end{eqnarray}
It is a local isometry if $\mathbb{H}^n\times\R$ is endowed with the metric $g_H-\mu^2 dt^2$ with $\mu=X_{n+1}.$ We may consider the Poincar\'e model of hyperbolic geometry instead of the hyperboloid model $\mathbb{H}^n$, that is the unit ball $B^n=\{x\in\R^n|\ |x|<1\}$ of $\R^n$ endowed with the hyperbolic metric 
$$g_H=\frac{4}{(1-|x|^2)^2}\sum_{i=1}^ndx_i^2.$$
By composing the map (\ref{intro map HnxR Q}) with the stereographic projection
\begin{eqnarray}
B^n&\rightarrow&\mathbb{H}^n\\
x&\mapsto&\left(\frac{2}{1-|x|^2}x_1,\ldots,\frac{2}{1-|x|^2}x_n,\frac{1+|x|^2}{1-|x|^2}\right)\nonumber
\end{eqnarray}
we get a local isometry $B^n\times\R\rightarrow\mathbb{H}^{n,1}$ if we equip $B^n\times\R$ with the metric
$$g=g_H-\mu^2dt^2\hspace{.5cm}\mbox{where}\hspace{.2cm}\mu=\frac{1+|x|^2}{1-|x|^2}.$$
The space $\widetilde{AdS}^{n,1}:=B^n\times\R$, with the metric $g$, is known as the Poincar\'e model of Anti-de Sitter geometry. It has the following nice property, which leads to a simple description of its causal structure: if we introduce the spherical metric 
$$g_S=\frac{4}{(1+|x|^2)^2}\sum_{i=1}^ndx_i^2$$
on $B^n$ (it is the metric induced on $B^n$ by a natural stereographic projection 
$$B^n\rightarrow (\mathbb{S}^n)^+:=\{(X_1,\ldots,X_{n+1}),\ \sum_{i=1}^{n+1}X_i^2=1,\ X_{n+1}>0\}$$ 
and the standard round metric on the sphere $\mathbb{S}^n$), we have $g_H=\mu^2 g_S$ and $g=\mu^2(g_S-dt^2).$ So the causal structure of $\widetilde{AdS}^{n,1}$ coincides with the causal structure of the product metric $g_S-dt^2$ on $B^n\times\R.$ We refer to \cite{BoSe1} for further properties of the models of Anti-de Sitter geometry.
\\

We say that an hypersurface $M$ of $\widetilde{AdS}^{n,1}$ is spacelike if the restriction to $M$ of the metric $g$ is riemannian. If $M$ is spacelike, let us denote by $N$ its upward unit normal vector field, i.e. the normal vector field of $M$ such that $g(N,N)=-1$ and $dt(N)>0$. If $\widetilde{\nabla}$ stands for the Levi-Civita connection of $\widetilde{AdS}^{n,1},$ the shape operator of $M$ is
$$S=\widetilde{\nabla}N:TM\rightarrow TM.$$
It is symmetric and its eigenvalues $\lambda_1,\ldots,\lambda_n$ are the principal curvatures of $M$. Let us consider the elementary symmetric functions of the principal curvatures of $M$
$$\sigma_m=\sum_{i_1<i_2<\cdots <i_m}\lambda_{i_1}\lambda_{i_2}\ldots\lambda_{i_m},$$
for $m=1,2,\ldots,n.$ The symmetric function $\sigma_1$ is the mean curvature, $\sigma_n$ is the Gauss-Kronecker curvature and $\sigma_2$ is linked to the intrinsic scalar curvature $Scal_M$ of $M$ by
$$Scal_M=-2\sigma_2-n(n-1).$$
A spacelike hypersurface of $\widetilde{AdS}^{n,1}$ is locally the graph of a smooth function $u:\Omega\subset B^n\rightarrow\R$ such that $|\nabla^Su|<1$ on $\Omega$, where $\nabla^Su$ denotes the gradient of $u$ with respect to the metric $g_S.$ If $e_1^0,\ldots, e_n^0$ is the canonical basis of $\R^n$, the basis $e_i=1/\lambda\ e_i^0,\ i=1,\ldots,n$ with $\lambda=\frac{2}{1+|x|^2}$ is an orthonormal basis of $(B^n,g_S).$ Defining, for $x\in B^n,$ $p\in B(0,1)\subset\R^n$ and $q\in S_n(\R)$ (which is the set of $n\times n$ symmetric matrices with real coefficients), the matrices
$$A(p)=\left(\delta_{ij}-p_ip_j\right)_{i,j}$$
whose inverse is
$$A(p)^{-1}=\left(\delta_{ij}+\frac{p_ip_j}{1-|p|^2}\right)_{i,j}$$
and
\begin{equation}\label{intro def aij xpq}
(a_{ij}(x,p,q))_{ij}:=\frac{1}{\mu(x)\sqrt{1-|p|^2}}\left(A(p)^{-1}\ q\ +d(\log\mu)_x\left(\sum_{s=1}^np_s e_s\right)\ I\right),
\end{equation}
the matrix in $(e_1,\ldots,e_n)$ of the shape operator of the graph of $u$ (using the natural chart $x\mapsto (x,u(x))$ of $M=\mbox{graph}(u)$) is
\begin{equation}\label{intro def Su}
S[u]=\left(a_{ij}(x,\nabla^Su,\nabla^Sdu)\right)_{ij}.
\end{equation}
Here and below, and with a slight abuse of notation, we use $p=\nabla^S u\in B(0,1)\subset\R^n$ and $q=\nabla^Sd u\in S_n(\R)$ to denote the components of the gradient and the hessian of $u$ in the orthonormal basis $(e_i)_{1\leq i\leq n}.$ The computations are carried out in Appendix \ref{appendix computations}. Setting
\begin{equation}\label{def Hmxpq intro}
\mathcal{H}_m(x,p,q):=\frac{m!(n-m)!}{n!}F_m((a_{ij}(x,p,q))_{ij}),
\end{equation}
where $F_m(A)=\sum_I A_{II}$ is the sum over all the multi-indices $I=(i_1,\ldots,i_m)$ with $1\leq i_1<\cdots<i_m\leq n$ of the determinants of the submatrices with rows and columns $I$ obtained from the matrix $A$, $\mathcal{H}_m(x,p,q)$ is the (normalized) $m^{th}$ symmetric function of the eigenvalues of $(a_{ij}(x,p,q))_{ij},$ and the $m^{th}$ curvature of the graph of $u$ is 
\begin{equation}\label{intro Hm=H}
\mathcal{H}_m[u]:=\mathcal{H}_m(x,\nabla^S u,\nabla^Sdu).
\end{equation}

Motivated by the recent classification of entire hypersurfaces with constant scalar curvature in Minkowski space \cite{BS} and its possible extension to Anti-de Sitter geometry, we are interested here in the Dirichlet problem for the equation of the prescribed scalar curvature in $\widetilde{AdS}^{n,1}.$ Let $\Omega$ be an open subset with compact closure in $B^n.$ We consider the Dirichlet problem
\begin{eqnarray}
\mathcal{H}_m[u]&=&H\hspace{.3cm}\mbox{ in }\Omega,\label{Dirichlet problem}\\
u&=&\varphi\hspace{.3cm}\mbox{ on }\partial\Omega\nonumber
\end{eqnarray}
where $H:\overline{\Omega}\rightarrow\R$ positive and $\varphi:\overline{\Omega}\rightarrow\R$ spacelike (i.e. such that $\sup_{\overline{\Omega}}|\nabla^S\varphi|<1$) are given. In order to consider an elliptic problem, we seek a solution in the class of admissible functions
$$K_m=\{u:\overline{\Omega}\rightarrow\R\ C^2,\ \mbox{spacelike such that}\ \mathcal{H}_k[u]>0\mbox{ for }k=1,\ldots, m\}.$$
It means that at every point of the graph of an admissible function $u$ the principal curvatures $\lambda_1,\ldots,\lambda_n$ belong to the cone
$$\Gamma_m=\{(\lambda_1,\ldots,\lambda_n)\in\R^n:\ \sigma_k>0,\ k=1,\ldots,m\}.$$
On $K_m$ the linearized operator of $\mathcal{H}_m$ is elliptic and the operator $\mathcal{H}_m^{\frac{1}{m}}$ is concave with respect to the second derivatives, i.e. the properties
\begin{equation}\label{intro Hm elliptic Km}
\sum_{i,j}\frac{\partial\mathcal{H}_m}{\partial q_{ij}}[u]\ \xi_i\ \xi_j>0
\end{equation}
for all $\xi\in\R^n\backslash\{0\}$ and
\begin{equation}\label{intro Hm concave Km}
\sum_{i,j,\ k,l}\frac{\partial^2\mathcal{H}_m^{\frac{1}{m}}}{\partial q_{ij}\partial q_{kl}}[u]\ \xi_{ij}\ \xi_{kl}\leq 0
\end{equation}
for all $(\xi_{ij})_{ij}\in S_n(\R)$ hold; this is a consequence of the algebraic properties of the operator $\mathcal{H}_m$ recalled in Appendix \ref{appendix properties Hm}. 

We will need to restrict our study to the case $m=2,$ which corresponds to the prescription of the scalar curvature; the main reason is the lack of a maximum principle for the second derivatives of a solution if $2<m<n.$ This is also an obstacle to the solution of the general problem of the prescription of $\mathcal{H}_m$ in Minkowski space \cite{Bay1,Ur}. We will nevertheless consider a general value of $m$ when we derive the required a priori estimates, when it is possible, and restrict to the case $m=2$ otherwise.

The next proposition shows that it is necessary to suppose that $\Omega$ is convex with respect to the spherical metric $g_{S}$ and the hyperbolic metric $g_H:$
\begin{prop}\label{prop intro necessary cond}
If there exists an admissible function which is constant on $\partial\Omega,$ then, at every point of $\partial\Omega,$ the principal curvatures $\kappa_1,\ldots,\kappa_{n-1}$ of $\partial\Omega$ with respect to the hyperbolic metric $g_H$ of $B^n$ belong to $\Gamma_{m-1},$ that is satisfy
\begin{equation}\label{dp bord admissible}
\sigma_k(\kappa_1,\ldots,\kappa_{n-1})>0\ \ \mbox{for } k=1,\ldots,m-1.
\end{equation}
If every spacelike totally geodesic boundary data extends to an admissible function, then $\Omega$ is necessarily convex with respect to the metric $g_S$. If moreover the center $0$ of $B^n$ belongs to $\Omega$, then $\Omega$ is also necessarily convex with respect to the metric $g_H.$
\end{prop}
The proof will be carried out in Section \ref{section necessary}. In view of that result it is natural to suppose that $\Omega\subset B^n$ is strictly convex with respect to the metric $g_{S}$ and contains the center 0 of $B^n.$ We will also consider a Dirichlet data $\varphi:\overline{\Omega}\rightarrow\R$ whose graph $M_{\varphi}$ is a spacelike hypersurface with boundary in $\widetilde{AdS}^{n,1}=B^n\times\R$ satisfying the following convexity assumptions:
\\
\\(H1) $M_{\varphi}$ is strictly convex: all its principal curvatures in $\widetilde{AdS}^{n,1}$ with respect to the upward unit normal are positive;
\\
\\(H2) $\partial M_{\varphi}$ is strictly convex: denoting by $II_{\partial M_{\varphi}}$ the second fundamental form of $\partial M_{\varphi}$ in $\widetilde{AdS}^{n,1}$, we suppose that $\langle II_{\partial M_{\varphi}},n\rangle$ is positive definite for every \emph{inward-directed} spacelike vector $n$ normal to $\partial M_{\varphi}$. 
\\

Inward-directed spacelike vectors normal to $\partial M_{\varphi}$ are defined as follows: at every point of $\partial M_{\varphi}$, the normal space $(T\partial M_{\varphi})^{\perp}$ is two-dimensional and equipped with the lorentzian metric $\langle.,.\rangle=g_{|(T\partial M_{\varphi})^{\perp}}$, and its set of spacelike vectors has two components; we say that a spacelike vector $n\in(T\partial M_{\varphi})^{\perp}$ is inward-directed if it belongs to the component of the inner unit vector $N_1$ normal to $\partial M_{\varphi}$ and tangent to $M_{\varphi}$ at that point, that is if it satisfies $\langle n,N_1\rangle>0$. 
\\

The main result of the paper is the following:
\begin{thm} \label{dp main thm}
Let $\Omega$ be an open subset with compact closure in $B^n$, with $\partial\Omega\in C^{2,\alpha},$ $\alpha\in (0,1)$. Assume moreover that $\Omega$ is strictly convex with respect to the metric $g_S$ and contains the center 0 of $B^n.$ Let $H:\overline{\Omega}\rightarrow\R$ be a positive function of class $C^{2,\alpha}$. Then, for every spacelike function $\varphi:\overline{\Omega}\rightarrow\R$ belonging to $C^{4,\alpha}$ and whose graph satisfies the convexity conditions (H1) and (H2) above, the Dirichlet problem (\ref{Dirichlet problem}) for $m=2$ admits a unique admissible solution $u\in C^{4,\alpha}.$ 
\end{thm}
\begin{rem}
(1) If $\Omega$ is strictly convex with respect to $g_S$ and contains the center $0$ of $B^n$ then (H2) holds for every totally geodesic boundary data $M_{\varphi}$. This is proved in Lemma \ref{app lem tt geod H2} in the appendix. In particular, if $\varphi:\overline{\Omega}\rightarrow\R$ is spacelike, strictly convex and constant on the boundary, its graph $M_{\varphi}$ satisfies (H1) and (H2).  
\\(2) Uniqueness is a direct consequence of the following comparison principle. We omit the proof, which is classical.
\begin{prop}
Let $\Omega$ be an open subset with compact closure in $B^n$ and let $u,v\in C^2(\overline{\Omega})$ be two spacelike functions. Let us assume that $v$ is admissible, $\mathcal{H}_m[v]\geq \mathcal{H}_m[u]$ in $\Omega$ and $v\leq u$ on $\partial\Omega.$ Then $v\leq u$ on $\overline{\Omega}.$ Moreover, if $v(x_0)=u(x_0)$ at an interior point $x_0\in\Omega$ then $v\equiv u$ in $\overline{\Omega}$ (the strong comparison principle).
\end{prop}
\noindent (3) From elliptic regularity theory, if $\partial\Omega\in C^{\infty},$ $\varphi\in C^{\infty}(\overline{\Omega})$ and $H\in C^{\infty}(\overline{\Omega})$, a solution $u$ belongs to $C^{\infty}(\overline{\Omega}).$
\\(4) The method of resolution relies on the obtention of a priori estimates of a solution, see Section \ref{section method C0}. In the paper, the global gradient estimate and the mixed second derivatives estimates at the boundary are in fact obtained for all the values of $m$; the hypothesis $m=2$ is required for the maximum principle for the second derivatives and also for the estimate of the double normal derivatives at the boundary. The same restrictions appear in previous works concerning the Dirichlet problem in Minkowski space, see \cite{Bay1} and \cite{Ur}.
\\(5) In Minkowski space $\R^{n,1}$ it is necessary to suppose that the open set $\Omega$ is convex in $\R^n$. It is also natural to suppose that $M_{\varphi}$ is strictly convex; nevertheless, the  hypothesis (H2) on the convexity of $\partial M _{\varphi}$ is not required, see \cite{Bay1}. In the present paper (H2) is needed for the gradient estimate on the boundary, see Section \ref{section C1}. There is indeed an important difference between the two Dirichlet problems: in $\widetilde{AdS}^{n,1}$ the vertical curves $t\mapsto (x,t)\in\ \widetilde{AdS}^{n,1}=B^n\times\R$ are not geodesics if $x\in B^n$ and $x\neq 0,$ and the center $0$ of $B^n$ plays a special role (as it also appears in the statement of Theorem \ref{dp main thm}). To overcome that difficulty, a more geometric Dirichlet problem in Anti-de Sitter geometry is possible, for the following graphs: if $u:\Omega\rightarrow\R$ is a function defined on an open subset $\Omega$ of a totally geodesic hypersurface $\mathbb{H}^n,$ its graph in $\widetilde{AdS}^{n,1}$ may be defined following the unit speed geodesics normal to $\mathbb{H}^n,$ on a proper time given by $u(x)$ for every $x\in\Omega.$ Nevertheless, entire hypersurfaces of $\widetilde{AdS}^{n,1}$ are not always entire graphs in that sense, what would restrict the applicability of the Dirichlet problem to the construction of entire hypersurfaces with prescribed scalar curvature: this is due to the fact that $\widetilde{AdS}^{n,1}$ is not globally geodesically convex, see for instance \cite{BoSe1}. 
\end{rem}
Let us quote some papers related to the Dirichlet problem, in lorentzian geometry. In Minkowski space, the Dirichlet problem for the equation of prescribed mean curvature was studied in \cite{BaSi}, for the equation of prescribed Gauss curvature in \cite{De}, \cite{BGuan} and \cite{AMLi}, and for the equation of prescribed scalar curvature in \cite{Bay1} and \cite{Ur}. The Dirichlet problem for the equation of prescribed mean curvature in spacetimes was studied in \cite{Bar}. An important motivation of the paper is the construction (and the classification) of entire hypersurfaces with constant scalar curvature in Anti-de Sitter space. Let us mention the following complete classifications (and refer to the bibliographies of these papers for other important contributions): in Minkowski space, entire hypersurfaces of constant mean curvature were classified in \cite{BoSeSmi2}, of constant Gauss curvature in dimension 3 in \cite{BoSeSmi1} and of constant scalar curvature in \cite{BS}. In Anti-de Sitter space, entire hypersurfaces of constant mean curvature were classified in \cite{Tre} and entire surfaces with constant Gauss curvature in dimension 3 in \cite{BoSe2}. Let us finally mention that convex entire hypersurfaces with prescribed curvature in de Sitter space were studied in \cite{SX}.
\\

The paper is organized as follows: we prove in Section \ref{section necessary} the necessity of the convexity assumption on the domain and we describe in Section \ref{section method C0} the method of resolution and the $C^0$ estimate. We then obtain the other required a priori estimates: the $C^1$ estimate is obtained in Section \ref{section C1} and the $C^2$ estimate in Section \ref{section C2}. We finally gather useful auxiliary results at the end of the paper: we prove basic formulas on second fundamental forms in Appendix \ref{appendix computations}, we gather algebraic properties of the operator $\mathcal{H}_m$ in Appendix \ref{appendix properties Hm} and we prove two elementary lemmas on spacelike and equidistant hypersurfaces of $\widetilde{AdS}^{n,1}$ in Appendix \ref{appendix equidistant}.


\section{Necessary conditions on $\Omega$: proof of Proposition \ref{prop intro necessary cond}}\label{section necessary}
We show here that it is necessary to suppose that $\Omega$ is convex with respect to the metric $g_{S},$ and also with respect to the hyperbolic metric $g_H$ if moreover $0$ belongs to $\Omega,$ proving Proposition \ref{prop intro necessary cond}. We will denote by $\nabla^S$ (resp. $\nabla^H$) the Levi-Civita connection of $B^n$ equipped with the metric $g_S$ (resp. $g_H$). We fix a boundary-point $x_0\in\partial\Omega$ and, considering the metric $g_S$, an orthonormal basis $e_1,\ldots,e_{n-1},e_n$ of $T_{x_0}B^n$ such that $e_1,\ldots,e_{n-1}$ are principal directions of $\partial\Omega$ and $e_n$ is the inner unit normal vector at $x_0$. Let us first observe that the principal curvatures $\kappa_1,\ldots,\kappa_{n-1}$ of $\partial\Omega$ with respect to the hyperbolic metric $g_{H}$ are given in terms of the principal curvatures $\kappa'_1,\ldots,\kappa'_{n-1}$ with respect to $g_{S}$ by the formulas
\begin{equation}\label{relation courbures principales Sn Hn}
\kappa_i=\frac{1}{\mu}\left(\kappa'_i-(\log\mu)_n\right),\ i=1,\ldots, n-1,
\end{equation}
which imply the relation
\begin{equation}\label{II partial Omega gS gH}
II_{\partial\Omega}^H=\mu\left(II_{\partial\Omega}^S-(\log\mu)_n\ g_S\right)
\end{equation}
where $II_{\partial\Omega}^H$ and $II_{\partial\Omega}^S$ stand for the second fundamental forms of $\partial\Omega$ with respect to the metrics $g_H$ and $g_S$ respectively. Let us prove (\ref{relation courbures principales Sn Hn}): $\partial\Omega$ is locally the graph $x_n=f(x')$ of a function $f:T_{x_0}\partial\Omega\rightarrow\R$ such that $f(0)=0$ and $df_0=0,$ and we set $\rho(x',x_n)=f(x')-x_n.$ Since the second fundamental form of $\partial\Omega$ with respect to $g_{S}$ is given by the restriction of $\frac{1}{|\nabla^S\rho|_S}\nabla^S d\rho$ to the boundary, and since $|\nabla^S\rho|_S=1$ at $x_0,$ we have $\kappa'_i=\nabla^S d\rho\ (e_i,e_i)$ for $i=1,\ldots,n-1.$ We may easily compare $\nabla^{H}$ and $\nabla^S$ since $g_H=\mu^2g_S,$ and then compare the hessians of $\rho$ with respect to the two metrics. We obtain: for all $X,Y\in T_{x_0}(\partial\Omega),$
$$\nabla^Hd\rho(X,Y)=\nabla^S d\rho(X,Y)+\frac{1}{\mu}d\rho(\nabla^S\mu)\ g_S(X,Y).$$
This implies that, for all $i,j=1,\ldots,n-1,$ $\nabla^{H}d\rho(e_i,e_j)=0$ if $i\neq j$, so that, since $II_{\partial\Omega}^H=\frac{1}{|\nabla^H\rho|_H}\ \nabla^Hd\rho$, $e_1,\ldots,e_{n-1}$ are also principal directions of $\partial\Omega$ with respect to $g_{H}.$ Moreover, since $\nabla^H\rho=1/\mu^2\ \nabla^S\rho$ with $|\nabla^S\rho|_S=1$ we have $|\nabla^H\rho|_H=1/\mu$ and
$$II_{\partial\Omega}^H=\mu\ \nabla^Hd\rho=\mu\nabla^Sd\rho+d\rho(\nabla^S\mu)\ g_S.$$ 
Since $g_{H}(e_i,e_i)=\mu^2$, we deduce that
$$\kappa_i=II_{\partial\Omega}^H\left(\frac{e_i}{\mu},\frac{e_i}{\mu}\right)=\frac{1}{\mu}\left(\kappa'_i+\frac{1}{\mu}\rho_n\mu_n\right)=\frac{1}{\mu}\left(\kappa'_i-\frac{1}{\mu}\mu_n\right),$$
which is (\ref{relation courbures principales Sn Hn}). Let $\varphi:\overline{\Omega}\rightarrow\R$ be a spacelike function whose graph is totally geodesic in $\widetilde{AdS}^{n,1}.$ By (\ref{lem expr h u}) in Appendix \ref{appendix computations} it satisfies in every orthonormal basis of $(B^n,g_S)$
$$\left(\varphi_{ij}\right)_{1\leq i,j\leq n}=-d(\log\mu)(\nabla^S\varphi)\ \left(\delta_{ij}-\varphi_i\varphi_j\right)_{1\leq i,j\leq n}$$
where the gradient and the second covariant derivative of $\varphi$ are with respect to the metric $g_S.$ We moreover assume that $\varphi$ is such that  $\nabla^S \varphi(x_0)=\varphi_1e_1$ (i.e. $\varphi_i=0$ for $i\geq 2$) and that there exists an admissible function $u:\overline{\Omega}\rightarrow\R$ such that $u=\varphi$ on $\partial\Omega.$ We have, at $x_0$,
$u_i=\varphi_i=\varphi_1\delta_{1i}$ for $i=1,\ldots,n-1$ and 
\begin{eqnarray}
\left(u_{ij}\right)_{1\leq i,j\leq n-1}&=&-u_n\left(\rho_{ij}\right)_{1\leq i,j\leq n-1}+\left(\varphi_{ij}\right)_{1\leq i,j\leq n-1}\nonumber\label{expr uij diag}\\
&=&-u_n\ \mbox{diag}(\kappa'_1,\ldots,\kappa'_{n-1})-d(\log\mu)(\nabla^S\varphi)\ \mbox{diag}(1-\varphi_1^2,1,\ldots,1)\label{expr uij diag}
\end{eqnarray}
where $\kappa'_1,\ldots,\kappa'_{n-1}$ are the principal curvatures of $\partial\Omega$ with respect to the metric $g_{S}.$ Note in particular that the matrix $\left(u_{ij}\right)_{1\leq i,j\leq n-1}$ is diagonal. We will use the following result: let $F_k(A)$ be the sum of the $k\times k$ principal minors of a matrix $A\in M_n(\R)$ and suppose that $p=p'+p_n e_n\in B(0,1)\subset\R^n$ with $p'\in\R^{n-1}\times\{0\}$, and $q\in S_n(\R);$ then the coefficient $A_{k}$ of $q_{nn}$ in the expression $\mathcal{F}_k(p,q):=F_k(A^{-1}(p)q)$ is given by
$$A_k=\frac{1-|p'|^2}{1-|p|^2}\mathcal{F}_{k-1}(p',q')$$
where $q'$ is the restriction of the quadratic form $q$ to $\R^{n-1}\times\{0\}$. Since ellipticity implies that $A_k=\frac{\partial\mathcal{F}_k}{\partial q_{nn}}(p,q)>0$ if $q\in\Gamma_k(p):=\{q\in S_n(\R):\ \mathcal{F}_i(p,q)>0,\ i=1,\ldots,k\},$ we deduce that $\mathcal{F}_{k-1}(p',q')>0$ in that case. Since $u$ is admissible, 
$$p=(u_i)_i=(\varphi_1,0,\ldots,0,u_n)\hspace{.3cm}\mbox{and}\hspace{.3cm}q=\left(u_{ij}+d(\log\mu)(\nabla^S u)\ (\delta_{ij}-u_iu_j)\right)_{ij}$$
are such that $q\in \Gamma_k(p)$ for $k=1,\ldots,m,$ which implies that $\mathcal{F}_{k-1}(p',q')>0$ for $k=1,\ldots,m.$ Now at $x_0$ the matrix $A^{-1}(p')q'$ is diagonal with diagonal entries
$$\frac{1}{1-\varphi_1^2}u_{11}+d(\log\mu)(\nabla^S u),\hspace{.5cm}u_{ii}+d(\log\mu)(\nabla^S u),\ 2\leq i\leq n-1.$$
Replacing $u_{ii}$ using (\ref{expr uij diag}) and since $d(\log\mu)(\nabla^S u)-d(\log\mu)(\nabla^S\varphi)=u_n\ (\log\mu)_n$ the condition $\mathcal{F}_{k-1}(p',q')>0$ readily implies (using (\ref{relation courbures principales Sn Hn}) again) that
$$(-u_n)^{k-1}\mu^{k-1}\left(\frac{1}{\mu}\left(\frac{1}{1-\varphi_1^2}\kappa'_1-(\log\mu)_n\right)\sigma_{k-2}(\kappa_2,\ldots,\kappa_{n-1})+\sigma_{k-1}(\kappa_2,\ldots,\kappa_{n-1})\right)>0$$
for $k=2,\ldots,m$. Now, since $\mathcal{H}_1[u]>0$ and $\mathcal{H}_1[\varphi]=0,$ we have $u\leq\varphi$ by comparison, and $u_n\leq\varphi_n=0$ at $x_0,$ so $u_n<0$ and
\begin{equation}\label{ineg phi1 kappa i}
\frac{1}{\mu}\left(\frac{1}{1-\varphi_1^2}\kappa'_1-(\log\mu)_n\right)\sigma_{k-2}(\kappa_2,\ldots,\kappa_{n-1})+\sigma_{k-1}(\kappa_2,\ldots,\kappa_{n-1})>0.
\end{equation}
For $\varphi_1=0$ we deduce from (\ref{relation courbures principales Sn Hn}) that $\sigma_{k-1}(\kappa_1,\ldots,\kappa_{n-1})>0$ for $k=2,\ldots,m,$ that is $\partial\Omega\in\Gamma_{m-1}.$ Since $\frac{\partial\sigma_{k-1}}{\partial\kappa_1}$ is positive on $\Gamma_{k-1},$ it implies that $\sigma_{k-2}(\kappa_2,\ldots,\kappa_{n-1})>0$ for $k=2,\ldots,m$ and, using (\ref{ineg phi1 kappa i}) again, we deduce that $\kappa'_1\geq 0$: if $\kappa'_1$ were strictly negative, (\ref{ineg phi1 kappa i}) would be impossible if $1/(1-\varphi_1^2)$ is chosen large. Since the principal direction $e_1$ was arbitrary, we deduce that $\Omega$ is convex with respect to $g_S.$ Finally, if $0\in\Omega$ then $\mu_n\leq 0$ by Lemma \ref{app mun negatif} in the appendix, and (\ref{relation courbures principales Sn Hn}) implies that $\Omega$ is also convex with respect to $g_H.$


\section{Method of resolution and the $C^0$ estimate}\label{section method C0}
We consider the family of Dirichlet problems
\begin{eqnarray}
\mathcal{H}_m[u]&=&sH+(1-s)\mathcal{H}_m[\varphi]\ \ \mbox{in}\ \Omega\label{eqn Hm s}\\
u&=&\varphi\ \ \mbox{on}\ \partial\Omega\nonumber
\end{eqnarray}
for $s\in [0,1].$ Let
$$K_m=\{u\in C^{2,\alpha}(\overline{\Omega}):\ \sup_{\overline{\Omega}}|\nabla^Su|<1,\ \mathcal{H}_k[u]>0,\ k=1,\ldots,m\}$$
be the set of admissible functions of class $C^{2,\alpha}$ and consider the set
$$S=\{s\in[0,1]:\ (\ref{eqn Hm s})\mbox{ admits a solution } u\in K_m\}.$$
The aim is to show that $1\in S.$ It will rely on the following simple observations:
\\\textbf{a.} $0\in S$ since $\varphi$ belongs to $K_m.$
\\\textbf{b.} $S$ is open in $[0,1]$ by the implicit functions theorem and the theory of linear elliptic equations. 
\\\textbf{c.} $S$ is closed in $[0,1]$ if the following estimates hold: there exist $\theta\in(0,1)$ and $M\geq 0$ such that, for all $s\in [0,1]$ and all $u\in K_m$ solution of (\ref{eqn Hm s}), we have 
\begin{equation}\label{est grad meth res}
\sup_{\overline{\Omega}}|\nabla^Su|\leq 1-\theta
\end{equation}
and
\begin{equation}\label{est C2alpha meth res}
\|u\|_{2,\alpha,\overline{\Omega}}\leq M.
\end{equation}
Here $\|.\|_{2,\alpha,\overline{\Omega}}$ denotes the natural $C^{2,\alpha}$ norm with respect to the metric $g_S$ on $\overline{\Omega}.$ By the Evans-Krylov and the Krylov-Safonov $C^{2,\alpha}$ estimates, (\ref{est C2alpha meth res}) follows from the gradient estimate (\ref{est grad meth res}) and the $C^0$ and $C^2$ estimates
\begin{equation}\label{est C0 C2 meth res}
\sup_{\overline{\Omega}}|u|\leq M_0\hspace{.5cm}\mbox{and}\hspace{.5cm}\sup_{\overline{\Omega}}|\nabla^Sdu|\leq M_2.
\end{equation}
So $S=[0,1]$ and 1 belongs to $S,$ i.e. the Dirichlet problem (\ref{Dirichlet problem}) is solvable, if the estimates (\ref{est grad meth res}) and (\ref{est C0 C2 meth res}) hold.
\\

Let us finally note that the $C^0$ estimate is straightforward. Let us consider $x\in\Omega,$ $y\in\partial\Omega$ and a geodesic $\gamma:[0,d]\rightarrow\overline{\Omega}$ with respect to $g_S$ such that $\gamma(0)=x,$ $\gamma(d)=y
$ and $g_S(\gamma',\gamma')=1$, so that $d=d_S(x,y)$. Since $g_S(\nabla^Su,\nabla^Su)\leq 1,$ we have by the Cauchy-Schwarz inequality
$$|u(y)-u(x)|=\left|\int_0^d\frac{d}{dt}u(\gamma(t))dt\right|=\left|\int_0^dg_S(\nabla^Su(\gamma(t)),\gamma'(t))dt\right|\leq \int_0^d1dt=d$$
which implies that
\begin{equation}\label{C0 u phi d}
\varphi(y)-d\leq u(x)\leq\varphi(y)+d.
\end{equation}
We deduce the following $C^0$ estimate: for all $x\in\Omega,$
$$\inf_{\partial\Omega}\varphi-diam_S(\Omega)\leq u(x)\leq \sup_{\partial\Omega}\varphi+diam_S(\Omega)$$
where $diam_S(\Omega)\in[0,\pi)$ stands for the diameter of $\Omega$ with respect to the metric $g_S.$
\section{The $C^1$ estimate}\label{section C1}
The aim of the section is to obtain the gradient estimate (\ref{est grad meth res}). We first apply a maximum principle to reduce the estimate to an estimate of the gradient on the boundary, and then obtain the estimate on the boundary using barriers. We first introduce some notation.
\subsection{Notation}
We will denote by $\widetilde{\nabla}$ the covariant derivative of $\widetilde{AdS}^{n,1}=B^n\times\R$ equipped with the metric $g=\mu^2(g_S-dt^2),$ and, if $M$ is a spacelike hypersurface of $\widetilde{AdS}^{n,1},$ by $\nabla$ the covariant derivative of $M$ equipped with the metric induced by $g.$ We will denote by $g_{ij}$ and $h_{ij}$ the metric and the second fundamental form of $M$ and we will use the Einstein summation convention, and raise or lower indices of tensors using the metric $(g_{ij}).$ The components of the shape operator of $M$ are thus denoted by $h^i_j,$ and if $F(A)$ is the sum of the principal minors of order $m$ of the matrix $A,$ we set
$$F_i^j:=\frac{\partial F}{\partial h^i_j}\left((h^i_j)_{i,j}\right)\hspace{.5cm}\mbox{and}\hspace{.5cm}F^{ij}=g^{ik}F_k^j.$$
Note that if $(h^i_j)_{ij}$ is diagonal, so is $F_i^j$ and $F^i_ i=\frac{\partial\sigma_m}{\partial\lambda_i}=\sigma_{m-1,i}$, the symmetric function of order $m-1$ of the principal curvatures of the hypersurface other than $\lambda_i.$ We will use a semi-colon to denote the components of the covariant derivative of a tensor. For the sake of simplicity we will often use the symbols $\langle.,.\rangle$ and $|.|$ to denote the metric and its associated norm. Note that it may depend on the context: for instance, on $B^n,$ we will use $|\nabla^Su|$ and $|\nabla^Hu|$ to denote the norm of the gradient of $u$ with respect to the metrics $g_S$ and $g_H;$ in the first case $|.|$ is the norm associated to $g_S$, and in the second case $|.|$ is the norm associated to $g_H$. Finally, we will need the following formula, which may be obtained by direct computations using the Christoffel formulas: for $X=X'+X^{n+1}e_{n+1}\in T\widetilde{AdS}^{n,1}=\R^n\times\R,$
\begin{equation}\label{notation C1 dev cov en+1}
\widetilde{\nabla}_Xe_{n+1}=\frac{1}{\mu}(X'.\mu)\ e_{n+1}+X^{n+1}\mu\nabla^H\mu.
\end{equation}
Here and below $e_{n+1}$ is the vertical vector such that $dt(e_{n+1})=1$, i.e. $e_{n+1}$ is the last vector of the canonical basis of $\R^{n+1}$. We may deduce from that formula that if $Y=Y'+Y^{n+1}e_{n+1}\in \Gamma(T\widetilde{AdS}^{n,1})$ is a vector field such that $Y'$ is independent of $t$, then
\begin{equation}\label{notation C1 dev cov Y}
\widetilde{\nabla}_XY=\nabla^H_{X'}{Y'}+X^{n+1}Y^{n+1}\mu\nabla^H\mu+\frac{1}{\mu}\left(X^{n+1}\ Y'.\mu+Y^{n+1}\ X'.\mu\right)e_{n+1}.
\end{equation}

\subsection{Reduction to the gradient estimate on the boundary}
We consider the time-function $t=x_{n+1}:B^n\times\R\rightarrow\R$ and its gradient
$$T=\widetilde{\nabla} t=-\frac{1}{\mu^2}e_{n+1}.$$
Let $M=\mbox{graph}(u)$ be a spacelike hypersurface in $\widetilde{AdS}^{n,1},$ $N$ its upward unit normal and $\nu$ its angle function defined by $\nu=\langle T,N\rangle.$ The expression (\ref{expression N}) of $N$ implies that
\begin{equation}\label{notation est C1 expr nu}
\nu=\frac{1}{\mu}\frac{1}{\sqrt{1-|\nabla^Su|^2}}
\end{equation}
and
\begin{equation}\label{notation est C1 expr N}
N=\nu\left(\nabla^Su+e_{n+1}\right).
\end{equation}
We assume that $u$ is a solution of the Dirichlet problem (\ref{Dirichlet problem}) and we want to show that $\sup_M\nu\leq C_1$ where $C_1$ is a controlled constant which only depends on $\sup_M|u|$, $\sup_{\partial M}\nu$ and on other controlled quantities as $\mu,$ $H$ and their derivatives on the domain. We will need formulas on some geometric quantities on $M$. In the following we do intrinsic computations on the hypersurface $M,$ and we use the Levi-Civita connection $\nabla$ on $M$ in the formulas in Lemmas \ref{lem Gauss formula}, \ref{lem nablamunu} and \ref{hessianMmu} below.
\begin{lem}\label{lem Gauss formula}
The function $u=t_{|M}:M\rightarrow \R$ satisfies the Gauss formula
\begin{equation}\label{Gauss formula}
u_{ij}+\frac{1}{\mu}\left(u_i\mu_j+u_j\mu_i\right)=\nu h_{ij}
\end{equation}
where $\mu$ is regarded as a function on $M$ and $h_{ij}$ is the second fundamental form of $M.$
\end{lem}
\begin{proof}
If $e_\alpha^0,$ $\alpha=1,\ldots,n$ is the canonical basis of $\R^{n}$, the basis $e_{\alpha}=1/\lambda\ e_{\alpha}^0,$ $\alpha=1,\ldots, n$ is a natural orthonormal basis of $(B^n,g_S)$. Let $\xi=\sum_{\alpha=1}^{n+1}\xi^\alpha e_\alpha:M\rightarrow\widetilde{AdS}^{n,1}$ be the position function. By definition, the second fundamental form of $M$ is given by $\widetilde{\nabla}d\xi(X,Y)=h(X,Y)N$ for all $X,Y\in TM.$ Here $d\xi:TM\rightarrow T\widetilde{AdS}^{n,1}$ is a 1-form and $\widetilde{\nabla}$ is the natural connection on $\Gamma(T^*M\otimes T\widetilde{AdS}^{n,1}).$ Writing $d\xi=\sum_{\alpha}d\xi^{\alpha}e_{\alpha}$ we have
\begin{equation*}
\widetilde{\nabla}d\xi(X,Y)=\sum_\alpha\nabla d\xi^\alpha(X,Y)e_\alpha+\sum_\alpha d\xi^\alpha(Y)\widetilde{\nabla}_Xe_\alpha.
\end{equation*}
Taking the scalar product with $e_{n+1}$ and since $\langle N,e_{n+1}\rangle=-\mu^2\nu$ (by (\ref{notation est C1 expr N})) and $\xi^{n+1}=u$ we get
\begin{equation}\label{eqn pf Gauss formula}
\mu^2\nu h(X,Y)=\mu^2\nabla du(X,Y)-\sum_\alpha Y^\alpha\langle\widetilde{\nabla}_X e_\alpha,e_{n+1}\rangle.
\end{equation}
Now, we obtain from (\ref{notation C1 dev cov en+1}) the formulas 
\begin{equation}\label{C1 christoffel 1}
\langle\widetilde{\nabla}_Xe_{\alpha},e_{n+1}\rangle=-X^{n+1}\ \mu\ e_\alpha.\mu=-X.u\ \mu\ e_\alpha.\mu,\ \alpha=1,\ldots,n,
\end{equation}
and
\begin{equation}\label{C1 christoffel 2}
\langle\widetilde{\nabla}_Xe_{n+1},e_{n+1}\rangle=-\mu\ X.\mu,
\end{equation}
where we have used in (\ref{C1 christoffel 1}) that a tangent vector $X\in TM$ satisfies $X^{n+1}=dx_{n+1}(X)=X.u$. So the last term of (\ref{eqn pf Gauss formula}) is
\begin{eqnarray*}
\sum_\alpha Y^\alpha\langle\widetilde{\nabla}_X e_\alpha,e_{n+1}\rangle&=&\sum_{\alpha=1}^n Y^\alpha\langle\widetilde{\nabla}_X e_\alpha,e_{n+1}\rangle+Y^{n+1}\langle\widetilde{\nabla}_X e_{n+1},e_{n+1}\rangle\\
&=&-\sum_{\alpha=1}^n Y^\alpha X.u\ \mu\ e_\alpha.\mu-Y^{n+1}\mu\ X.\mu\\
&=&-\mu X.u\ Y.\mu-\mu Y.u\ X.\mu
\end{eqnarray*}
and the result follows.
\end{proof}
\begin{lem}\label{lem nablamunu}
The gradient of the function $\nu$ is given by
\begin{equation}\label{nablamunu}
\nabla(\mu\nu)=-\frac{1}{\nu}\langle\nabla\mu,\nabla u\rangle\nabla u+\mu S(\nabla u)
\end{equation}
where $S:TM\rightarrow TM$ is the shape operator of $M.$
\end{lem}
\begin{proof}
Since $\nu=\langle T,N\rangle,$ we have
\begin{equation}\label{eqn dnu pf nablamunu}
d\nu(X)=\langle\widetilde{\nabla}_XT,N\rangle+\langle T,\widetilde{\nabla}_XN\rangle
\end{equation}
for all $X\in TM$. We compute the first term on the right-hand side of (\ref{eqn dnu pf nablamunu}): since $T=-1/\mu^2\ e_{n+1}$ we easily get using (\ref{C1 christoffel 1}) and (\ref{C1 christoffel 2}) that
$$\widetilde{\nabla}_XT=-\frac{1}{\mu}X.u\ \nabla^H\mu+\frac{1}{\mu^3}X.\mu\ e_{n+1}$$
and using (\ref{notation est C1 expr N}) with $\nabla^Su=\mu^2\nabla^Hu$ we deduce that
\begin{eqnarray}
\langle\widetilde{\nabla}_XT,N\rangle&=&-\mu\nu X.u\ \langle\nabla^H\mu,\nabla^H u\rangle-\frac{\nu}{\mu}X.\mu\nonumber\\
&=&-\frac{1}{\mu\nu} X.u\ \langle\nabla\mu,\nabla u\rangle-\frac{\nu}{\mu}X.\mu.\label{dnu interm1}
\end{eqnarray}
We used in the last step the formula
\begin{equation}\label{pf lem grad nu eqn interm} 
\langle\nabla\mu,\nabla u\rangle=\mu^2\nu^2\langle\nabla^H\mu,\nabla^H u\rangle;
\end{equation}
on the left-hand side $\nabla\mu$ and $\nabla u$ belong to $TM$, and on the right-hand side $\nabla^H\mu$ and $\nabla^H u$ belong to $TB^n$ and are gradients computed with respect to the hyperbolic metric $g_H.$ We will prove that formula below, but we first compute the second term of the right-hand side of (\ref{eqn dnu pf nablamunu}) and conclude the proof of (\ref{nablamunu}): noticing that the orthogonal projection of $T$ on $TM$ is $T^{||}=\nabla u,$ we have
\begin{equation}\label{dnu interm2}
\langle T,\widetilde{\nabla}_XN\rangle=\langle T^{||},S(X)\rangle=\langle\nabla u,S(X)\rangle=\langle S(\nabla u),X\rangle,
\end{equation}
and (\ref{eqn dnu pf nablamunu}), together with (\ref{dnu interm1}) and (\ref{dnu interm2}), yields
$$\nabla\nu=-\frac{1}{\mu\nu}\langle\nabla\mu,\nabla u\rangle\nabla u-\frac{\nu}{\mu}\nabla\mu+S(\nabla u),$$
which gives (\ref{nablamunu}). We finally prove (\ref{pf lem grad nu eqn interm}). We first note that a vector $X\in TM$ is of the form $X=X'+\langle\nabla^Hu,X'\rangle\ e_{n+1}$ with $X'\in TB^n$. By composition with the natural projection $M\rightarrow B^n,$ a function $f:B^n\rightarrow\R$ may be considered as a function $f:M\rightarrow\R,$ and by the observation above its gradient may be written $\nabla f=(\nabla f)'+\langle\nabla ^H u,(\nabla f)'\rangle\ e_{n+1}$ where $(\nabla f)'$ belongs to $TB^n.$ We have, for all $X=X'+\langle\nabla^Hu,X'\rangle\ e_{n+1}\in TM,$
$$df(X')=df(X)=\langle\nabla f,X\rangle=\langle (\nabla f)',X'\rangle-\mu^2\langle\nabla ^H u,(\nabla f)'\rangle\langle\nabla^Hu,X'\rangle$$
which implies that 
$$\nabla^Hf=(\nabla f)'-\mu^2\langle\nabla^Hu,(\nabla f)'\rangle\nabla^H u.$$ 
Taking the scalar product with $\nabla^Hu$ and using that 
\begin{equation}\label{est C1 expr gradH mu}
\mu^2|\nabla^Hu|^2=1-1/{\mu^2\nu^2}
\end{equation} 
(as a consequence of (\ref{notation est C1 expr nu}), using that $\nabla^Su=\mu^2\nabla^Hu$ and where the norm of $\nabla^Hu$ is with respect to $g_H$) we deduce that $\langle(\nabla f)',\nabla^Hu\rangle=\mu^2\nu^2\langle\nabla^Hf,\nabla^Hu\rangle.$ So we have
$$(\nabla f)'=\nabla^Hf+\mu^4\nu^2\langle\nabla^Hf,\nabla^Hu\rangle\nabla^Hu$$
and deduce that 
\begin{equation}\label{pf lem grad nu eqn nabla f nabla H}
\nabla f=\nabla^Hf+\mu^2\nu^2\langle\nabla^Hf,\nabla^Hu\rangle\left(\mu^2\nabla^Hu+e_{n+1}\right)=\nabla^Hf+\mu^2\nu\langle\nabla^Hf,\nabla^Hu\rangle N.
\end{equation}
Using (\ref{pf lem grad nu eqn nabla f nabla H}) for $f=u$ and $\mu$, and computing using (\ref{notation est C1 expr N}) with $\nabla^Su=\mu^2\nabla^Hu$ we easily obtain
$$\langle\nabla u,\nabla\mu\rangle=\langle\nabla^H u,\nabla^H\mu\rangle\left(1+\mu^4\nu^2|\nabla^Hu|^2\right),$$
which implies (\ref{pf lem grad nu eqn interm}) and ends the proof of the lemma.
\end{proof}
\begin{rem}
Let us note the following estimate that we will use later: by (\ref{pf lem grad nu eqn nabla f nabla H}) we have
\begin{equation}\label{nabla f leq nablaH f}
|\nabla f|=\left(|\nabla^Hf|^2+\mu^4\nu^2|\langle\nabla^Hf,\nabla^Hu\rangle|^2\right)^{\frac{1}{2}}\leq \mu\nu|\nabla^Hf|;
\end{equation}
it is a consequence of the Schwarz inequality and (\ref{est C1 expr gradH mu}). In particular, the equality on the left-hand side with $f=u$ and formula (\ref{est C1 expr gradH mu}) imply that
\begin{equation}\label{C1 norm nabla u leq mu nu}
\nu^2-1\ \leq\ |\nabla u|^2=\nu^2-\frac{1}{\mu^2}\ \leq\ \nu^2.
\end{equation}

\end{rem}
\begin{lem}\label{hessianMmu}
The hessian of $\mu:M\rightarrow\R$ is given by 
\begin{equation}
\mu_{ij}=\mu (g_{ij}+u_i u_j)+\frac{1}{\nu}\langle\nabla\mu,\nabla u\rangle h_{ij}. 
\end{equation}
\end{lem}
\begin{proof}
For all $X,Y\in TM,$ since
$$\nabla_Xd\mu(Y)=X.d\mu(Y)-d\mu(\nabla_XY)\hspace{.5cm}\mbox{and}\hspace{.5cm}\widetilde{\nabla}_Xd\mu(Y)=X.d\mu(Y)-d\mu(\widetilde{\nabla}_XY)$$
we have
\begin{equation}\label{nabla2mueq1}
\nabla_Xd\mu(Y)-\widetilde{\nabla}_Xd\mu(Y)=-d\mu(\nabla_XY-\widetilde{\nabla}_XY)=h(X,Y)d\mu(N).
\end{equation}
We compute $\widetilde{\nabla}_Xd\mu(Y)$: we write $X=X'+X^{n+1}e_{n+1},$ $Y=Y'+Y^{n+1}e_{n+1},$ and assume for the computation that $e_{n+1}.Y'=0.$ Since $d\mu(Y)=d\mu(Y'),$ which is a function which does not depend on $x_{n+1}$, we have $X.(d\mu(Y))=X'.(d\mu(Y'))$ and get
$$\widetilde{\nabla}_Xd\mu(Y)-\nabla^H_{X'}d\mu(Y')=-d\mu\left((\widetilde{\nabla}_XY)'-\nabla^H_{X'}Y'\right);$$
since $(\widetilde{\nabla}_XY)'=\nabla^H_{X'}Y'+X^{n+1}Y^{n+1}\mu\nabla^H\mu$ (a consequence of (\ref{notation C1 dev cov Y})) we deduce that
\begin{eqnarray}
\widetilde{\nabla}_Xd\mu(Y)&=&\nabla^H_{X'}d\mu(Y')-X^{n+1}\ Y^{n+1}\ \mu|\nabla^H\mu|^2\nonumber\\
&=&\mu\langle X',Y'\rangle-X^{n+1}\ Y^{n+1}\ \mu(\mu^2-1)\nonumber\\
&=&\mu\langle X,Y\rangle+\mu X.u\ Y.u,\label{nabla2mueq2}
\end{eqnarray}
where we have also used the following properties of $\mu$ (obtained by direct computations in $(B^n,g_H)$, or using that $\mu=X_{n+1}$ in the model (\ref{intro hyperboloid model}) of $\mathbb{H}^n$): $\nabla^H\mu=x$, $|\nabla^H\mu|^2=\mu^2-1$ and $\nabla^H_{X'}d\mu(Y')=\mu\langle X',Y'\rangle.$ Since $d\mu(N)=\frac{1}{\nu}\langle\nabla\mu,\nabla u\rangle$ (by (\ref{notation est C1 expr N}) with $\nabla^Su=\mu^2\nabla^Hu$ and using (\ref{pf lem grad nu eqn interm})), (\ref{nabla2mueq1}) and (\ref{nabla2mueq2}) yield the result.
\end{proof}
\begin{prop}
Let $M=\mbox{graph}(u)$, where $u:\overline{\Omega}\rightarrow\R$ is an admissible solution of the Dirichlet problem (\ref{Dirichlet problem}), with prescribed curvature function $H:\overline{\Omega}\rightarrow\R.$ Then there exists a constant $C_1$ which only depends on $\sup_M|u|$, $\sup_{\partial M}\nu$, $|\mu|_{1,\overline{\Omega}},$ $\inf_{\overline{\Omega}}H$ and $|H|_{1,\overline{\Omega}}$ such that $\sup_M\nu\leq C_1.$
\end{prop}
The rest of the section is dedicated to the proof of that proposition. We will prove that there exist two constants $A$ and $K$ such that the function
$$\varphi=\nu\mu^A e^{Ku}$$
reaches its maximum at the boundary $\partial M.$ We still do computations on $M$, using the Levi-Civita connection on $M.$ We suppose by contradiction that $\varphi$ reaches its maximum at an interior point $x_0\in M.$ We have, at $x_0,$ $(\log\varphi)_i=0,$ that is
\begin{equation}\label{gradient extr condition}
(A-1)(\log\mu)_i+(\log(\mu\nu))_i+Ku_i=0
\end{equation}
for $i=1,\ldots,n$ (the extremum condition) and $F^{ij}(\log\varphi)_{ij}\leq 0,$ that is
\begin{equation}\label{gradient max condition}
(A-1)F^{ij}(\log\mu)_{ij}+F^{ij}(\log(\mu\nu))_{ij}+KF^{ij}u_{ij}\leq 0
\end{equation}
(the maximum condition). We compute successively each one of the three terms of (\ref{gradient max condition}). Let us fix an orthonormal basis of principal directions of $M$ at $x_0$ and note that $(F^{ij})_{ij}=(\sigma_{m-1,i}\delta_{ij})_{ij}$ is diagonal in that basis.
\begin{lem}
The first term of (\ref{gradient max condition}) is given by
\begin{equation}\label{gmct1}
F^{ii}(\log \mu)_{ii}=(n-m+1)\sigma_{m-1}+F^{ii}u_i^2+\frac{m}{\mu\nu}H\langle\nabla\mu,\nabla u\rangle-\frac{1}{\mu^2}F^{ii}\mu_i^2.
\end{equation}
\end{lem}
\begin{proof}
This is a direct consequence of Lemma \ref{hessianMmu}.
\end{proof}
\begin{lem}
The second term of (\ref{gradient max condition}) is computed thanks to the following formulas:
\begin{equation}\label{gmct2}\frac{1}{\mu\nu}F^{ij}(\mu\nu)_i(\mu\nu)_j=\frac{1}{\mu\nu^3}\langle\nabla\mu,\nabla u\rangle^2F^{ii}u_i^2-\frac{2}{\nu^2}\langle\nabla\mu,\nabla u\rangle F^{ii}\lambda_iu_i^2+\frac{\mu}{\nu}F^{ii}\lambda_i^2u_i^2;
\end{equation}
\begin{eqnarray}
F^{ij}(\mu\nu)_{ij}&=&\left(-\frac{1}{\mu\nu^3}\langle\nabla\mu,\nabla u\rangle^2-\frac{\mu}{\nu}(1+|\nabla u|^2)+\frac{1}{\mu\nu}|\nabla \mu|^2\right)F^{ii}u_i^2\label{gmct3}\\
&&+\frac{2}{\mu\nu}\langle\nabla\mu,\nabla u\rangle F^{ii}u_i\mu_i-2F^{ii}\lambda_iu_i\mu_i+\mu\nu F^{ii}\lambda_i^2\nonumber\\
&&-m\langle\nabla\mu,\nabla u\rangle H+\mu\langle\nabla H,\nabla u\rangle.\nonumber
\end{eqnarray}
\end{lem}
\begin{proof}
(\ref{gmct2}) is a direct consequence of Lemma \ref{lem nablamunu}. Let us prove (\ref{gmct3}). We have
$$(\mu\nu)_i=-\frac{1}{\nu}\langle\nabla\mu,\nabla u\rangle u_i+\mu h_{ik}u^k$$
and setting
$$B_{ij}=\left(-\frac{1}{\nu}\langle\nabla\mu,\nabla u\rangle u_i\right)_{;j}+\mu_j h_{ik}u^k$$
we have 
\begin{equation}\label{Fijmunupf}
F^{ij}(\mu\nu)_{ij}=F^{ij}B_{ij}+\mu(F^{ij}h_{ik;j}u^k+F^{ij}h_{ik}u^k_{j}).
\end{equation}
We first compute using (\ref{Gauss formula})
\begin{eqnarray*}
F^{ij}B_{ij}&=&\frac{1}{\nu^2}\langle\nabla\mu,\nabla u\rangle F^{ij}u_i\nu_j-\frac{1}{\nu}g_{kl}F^{ii}\left(\mu_i^ku^l+\mu^ku_i^l\right)u_i\\
&&+\frac{2}{\mu\nu}\langle\nabla\mu,\nabla u\rangle F^{ii}u_i\mu_i-m\langle\nabla\mu,\nabla u\rangle H+F^{ii}\lambda_iu_i\mu_i.
\end{eqnarray*}
By Lemma \ref{lem nablamunu} we have 
$$F^{ij}u_i\nu_j=-\frac{1}{\mu\nu}\langle\nabla\mu,\nabla u\rangle F^{ii}u_i^2+F^{ii}\lambda_iu_i^2-\frac{\nu}{\mu}F^{ii}u_i\mu_i,$$
and by Lemmas \ref{lem Gauss formula} and \ref{hessianMmu}
\begin{eqnarray*}
-\frac{1}{\nu}g_{kl}F^{ii}\left(\mu_i^ku^l+\mu^ku_i^l\right)u_i&=&\left(-\frac{\mu}{\nu}(1+|\nabla u|^2)+\frac{1}{\mu\nu}|\nabla\mu|^2\right)F^{ii}u_i^2-\frac{1}{\nu^2}\langle\nabla\mu,\nabla u\rangle F^{ii}\lambda_iu_i^2\\
&&+\frac{1}{\mu\nu}\langle\nabla\mu,\nabla u\rangle F^{ii}\mu_iu_i-F^{ii}\lambda_iu_i\mu_i.
\end{eqnarray*}
We finally get
\begin{eqnarray}
F^{ij}B_{ij}&=&\left(-\frac{1}{\mu\nu^3}\langle\nabla\mu,\nabla u\rangle^2-\frac{\mu}{\nu}(1+|\nabla u|^2)+\frac{1}{\mu\nu}|\nabla\mu|^2\right)F^{ii}u_i^2\nonumber\\
&&+\frac{2}{\mu\nu}\langle\nabla\mu,\nabla u\rangle F^{ii}u_i\mu_i-m\langle\nabla\mu,\nabla u\rangle H.\label{FijBijpf}
\end{eqnarray}
We then compute the last two terms in (\ref{Fijmunupf}): by the Codazzi equation we have
\begin{equation}\label{pf lem munuij eq1}
F^{ij}h_{ik;j}u^k=F^{ij}h_{ij;k}u^k=H_ku^k=\langle\nabla H,\nabla u\rangle,
\end{equation}
and (\ref{Gauss formula}) yields
\begin{equation}\label{pf lem munuij eq2}
\mu F^{ij}h_{ik}u^k_{j}=-2F^{ii}\lambda_iu_i\mu_i+\mu\nu F^{ii}\lambda_i^2.
\end{equation}
Inserting (\ref{FijBijpf}), (\ref{pf lem munuij eq1}) and (\ref{pf lem munuij eq2}) into (\ref{Fijmunupf}) we obtain the result (\ref{gmct3}). 
\end{proof}
\begin{lem}
The last term of (\ref{gradient max condition}) is given by
\begin{equation}\label{gmct4}
F^{ij}u_{ij}=-\frac{2}{\mu}F^{ii}u_i\mu_i+mH\nu.
\end{equation}
\end{lem}
\begin{proof}
This is a direct consequence of Gauss formula (\ref{Gauss formula}).
\end{proof}
Inserting formulas (\ref{gmct1}), (\ref{gmct2}), (\ref{gmct3}) and (\ref{gmct4}) in (\ref{gradient max condition}) we obtain
\begin{eqnarray}
\left(-\frac{2}{\mu^2\nu^4}\langle\nabla\mu,\nabla u\rangle^2-\frac{1}{\nu^2}(1+|\nabla u|^2)+\frac{1}{\mu^2\nu^2}|\nabla \mu|^2+ (A-1)\right)F^{ii}u_i^2&&\label{gmct5}\\
+(A-1)F^{ii}\frac{\mu_i^2}{\mu^2}+F^{ii}\lambda_i^2\left(1-\frac{u_i^2}{\nu^2}\right)+\frac{2}{\mu\nu^3}\langle\nabla\mu,\nabla u\rangle F^{ii}\lambda_iu_i^2&&\nonumber\\
+mHK\nu+(A-1)(n-m+1)\sigma_{m-1}+m\frac{(A-1)}{\mu\nu}H\langle\nabla\mu,\nabla u\rangle&\leq&C\nu\nonumber
\end{eqnarray}
where we have used the formula (which is a direct consequence of (\ref{nablamunu}) and (\ref{gradient extr condition}))
$$2\left(\frac{1}{\mu^2\nu^2}\langle\nabla\mu,\nabla u\rangle-\frac{K}{\mu}\right)F^{ii}u_i\mu_i-\frac{2}{\mu\nu}F^{ii}\lambda_iu_i\mu_i=2(A-1)F^{ii}\frac{\mu_i^2}{\mu^2}$$
and the bound (consequence of (\ref{nabla f leq nablaH f}))
$$\left|-\frac{m}{\mu\nu}\langle\nabla\mu,\nabla u\rangle H+\frac{1}{\nu}\langle \nabla H,\nabla u\rangle\right|\leq C\nu.$$
Let us see how to balance the bad term
\begin{equation}\label{gradient bad term}
\frac{2}{\mu\nu^3}\langle\nabla\mu,\nabla u\rangle \sum_iF^{ii}\lambda_iu_i^2
\end{equation}
of the second line of (\ref{gmct5}). Let us first note that we may assume that $\sum_iF^{ii}\lambda_iu_i^2<0$ and $\langle\nabla\mu,\nabla u\rangle>0$: we have
\begin{equation}\label{gmct5-6}
\sum_iF^{ii}\lambda_iu_i^2=\sum_i(\sigma_{m-1}-\lambda_i\sigma_{m-2,i})\lambda_iu_i^2=\sigma_{m-1}\sum_i\lambda_iu_i^2-\sum_i\sigma_{m-2,i}\lambda_i^2u_i^2
\end{equation}
with, by (\ref{nablamunu}) and the maximum condition (\ref{gradient extr condition}), 
\begin{equation*}
\frac{1}{\nu}\sum_i\lambda_iu_i^2=\frac{1}{\nu}\langle S(\nabla u),\nabla u\rangle=-\frac{(A-1)}{\mu}\langle\nabla\mu,\nabla u\rangle-\left(K-\frac{1}{\mu\nu^2}\langle\nabla\mu,\nabla u\rangle\right)|\nabla u|^2;
\end{equation*}
that last expression may be assumed to be negative if $K=K(A)$ is chosen sufficiently large (noticing that $|\langle\nabla\mu,\nabla u\rangle|\leq C\nu^2$ and $|\nabla u|^2\geq \nu^2-1$ by (\ref{C1 norm nabla u leq mu nu})), which implies that (\ref{gmct5-6}) is negative. We moreover consequently assume that $\langle\nabla\mu,\nabla u\rangle>0$ (if not, the term (\ref{gradient bad term}) would be non-negative and could be dismissed in (\ref{gmct5})). So, setting $J=\{j|\ \lambda_j<0\}$ we only have to balance the negative contribution
$$\frac{2}{\mu\nu^3}\langle\nabla\mu,\nabla u\rangle\sum_{j\in J}F^{jj}\lambda_ju_j^2.$$
Using that $|u_j|\langle\nabla\mu,\nabla u\rangle\leq\mu^2\nu^3$ (since $|\nabla u|\leq \nu$ by (\ref{C1 norm nabla u leq mu nu}) and $|\nabla\mu|\leq \mu\nu|\nabla ^H\mu|$ with $|\nabla^H\mu|=|x|_H\leq\mu$ by (\ref{nabla f leq nablaH f})) and the inequality $2ab\leq \varepsilon a^2+\frac{1}{\varepsilon}b^2$ for some $\varepsilon>0$ that will be fixed later, we get
\begin{eqnarray}
\frac{2}{\mu\nu^3}\langle\nabla\mu,\nabla u\rangle \sum_{j\in J}F^{jj}|\lambda_j|u_j^2&\leq&2\mu\sum_{j\in J}F^{jj}|\lambda_j||u_j|\nonumber\\
&\leq& \varepsilon\sum_{j\in J}F^{jj}\lambda_j^2+\frac{1}{\varepsilon}\mu^2\sum_{j\in J}F^{jj}u_j^2.\label{ineq bad term epsilon C1 max dirichlet}
\end{eqnarray}
The last term on the right-hand side will be balanced by the term $(A-1)F^{ii}u_i^2$ in (\ref{gmct5}) if $A$ is chosen sufficiently large. The first term on the right-hand side will be balanced by the term $F^{ii}\lambda_i^2\left(1-\frac{u_i^2}{\nu^2}\right)$ in (\ref{gmct5}) thanks to the following lemma:
\begin{lem}\label{lem epsilon C1 max dirichlet}
There exists $\varepsilon=\varepsilon(n)>0$ such that
\begin{equation}\label{eqn lem epsilon C1 max dirichlet}
\sum_iF^{ii}\lambda_i^2\left(1-\frac{u_i^2}{\nu^2}\right)\geq\varepsilon\ \sum_{j\in J}F^{jj}\lambda_j^2.
\end{equation}
\end{lem}
\begin{proof}
It relies on the following inequality of Ivochkina, Lin and Trudinger \cite[ineq. (26)]{LT} (see also (\ref{est C2b ineg ILT1}) below): 
if $\lambda_{j}\leq 0,$
$$\sum_{i\neq j}\sigma_{m-1,i}\lambda_i^2\geq c\sigma_{m-1,j}\lambda_j^2$$
where $c=c(n)$ is a small constant. That inequality in turn implies that, for all $j$ such that $\lambda_j\leq 0,$
$$\sum_iF^{ii}\lambda_i^2\left(1-\frac{u_i^2}{\nu^2}\right)\geq c F^{jj}\lambda_j^2.$$
See \cite[Lemma 4.6]{Bay2} for the proof. The sum of these inequalities for $j\in J$ gives the result with $\varepsilon=c/k$ where $k\geq 1$ is the cardinal of $J.$
\end{proof}
So the inequalities (\ref{ineq bad term epsilon C1 max dirichlet}) and (\ref{eqn lem epsilon C1 max dirichlet}) inserted in (\ref{gmct5}) yield
\begin{eqnarray}
\left(-\frac{2}{\mu^2\nu^4}\langle\nabla\mu,\nabla u\rangle^2-\frac{1}{\nu^2}(1+|\nabla u|^2)+\frac{1}{\mu^2\nu^2}|\nabla \mu|^2+ (A-1)-\frac{1}{\varepsilon}\mu^2\right)F^{ii}u_i^2&&\label{gmct6}\\
+(A-1)F^{ii}\frac{\mu_i^2}{\mu^2}+mHK\nu+(A-1)(n-m+1)\sigma_{m-1}+m\frac{(A-1)}{\mu\nu}H\langle\nabla\mu,\nabla u\rangle&\leq&C\nu.\nonumber
\end{eqnarray}
Finally, using that $1\leq \mu\leq c_1$ and $\left|\langle\nabla\mu,\nabla u\rangle\right|$, $|\nabla u|^2,$ $|\nabla \mu|^2\leq c_2\nu^2$ for some constants $c_1,c_2,$ we see that (\ref{gmct6}) is impossible if $A$ first, and then $K=K(A)$, are chosen sufficiently large.

\subsection{The gradient estimate on the boundary}
The lower gradient estimate on the boundary relies on the construction of a convenient lower barrier function. It is inspired by the construction in \cite{De}. We need to suppose here that $\Omega\subset B^n$ is strictly convex with respect to $g_{S}$ and contains the center $0$ of $B^n$ (and is therefore also strictly convex with respect to $g_H$), and that the Dirichlet data $\varphi:\overline{\Omega}\rightarrow\R$ is a spacelike function whose graph $M_{\varphi}=\{(x,\varphi(x)),\ x\in\overline{\Omega}\}$ satisfies the convexity assumptions (H1) and (H2). If $\mbox{Isom}^o(\widetilde{AdS}^{n,1})$ stands for the identity component of the isometry group of $\widetilde{AdS}^{n,1}=B^n\times\R,$ we say that $f\in\mbox{Isom}^o(\widetilde{AdS}^{n,1})$ is \emph{a chart adapted to an interior point $x_0\in M_{\varphi}\backslash\partial M_{\varphi}$} if $f(x_0)=(0,0)\in B^n\times\R$ and $df_{x_0}(T_{x_0}M_{\varphi})=T_0B^n\times\{0\}\simeq\R^n,$ and that it is \emph{a chart adapted to a boundary-point $x_0\in\partial M_{\varphi}$} if
\begin{equation}\label{cond f chart adapted boundary point}
f(x_0)=(0,0),\ df_{x_0}(T_{x_0}M_{\varphi})=\R^n\hspace{.3cm}\mbox{and}\hspace{.3cm}df_{x_0}(T_{x_0}\partial M_{\varphi})=\R^{n-1}\times\{0\}.
\end{equation}
The convexity assumptions on $\Omega,$ $M_{\varphi}$ and $\partial M_{\varphi}$ above guarantee the following result:
\begin{lem}\label{est C1 lem Omega convex}
For every chart $f$ adapted to a point of $M_{\varphi}$, the hypersurface $f(M_{\varphi})$ is a graph above a strictly convex subset $\overline{\Omega}_f$ of $B^n,$ with respect to the hyperbolic metric $g_H$; moreover $0$ belongs to $\overline{\Omega}_f$ and the graph is strictly convex with respect to the metric of $\widetilde{AdS}^{n,1}.$ 
\end{lem}
\begin{proof}
Since a chart $f$ adapted to a point $x_0\in M_{\varphi}$ is by definition an isometry of $\widetilde{AdS}^{n,1}$ such that $f(x_0)=(0,0)$, the facts that $0$ belongs to $\overline{\Omega}_f$ and that $f(M_{\varphi})$ is strictly convex with respect to the metric of $\widetilde{AdS}^{n,1}$ are obvious. We now prove that $\overline{\Omega}_f$ is strictly convex. Let us consider the set 
$$\mathcal{C}:=\left\{f\in \mbox{Isom}^o(\widetilde{AdS}^{n,1}) \mbox{ such that }f(x_0)=(0,0)\mbox{ for some }x_0\in M_{\varphi}\right\}$$
and its subset
$$\mathcal{C}_0:=\left\{f\in\mathcal{C}:\ \overline{\Omega}_f \mbox{ is strictly convex}\right\}.$$
$\mathcal{C}$ is a connected set of isometries of $\widetilde{AdS}^{n,1}$. This is a consequence of the following two observations: for every path $\gamma:[0,1]\rightarrow M_{\varphi}\subset\widetilde{AdS}^{n,1}$ there is a path $(f_t)_{t\in[0,1]}$ in $\mbox{Isom}^o(\widetilde{AdS}^{n,1})$ such that $f_t(\gamma(t))=(0,0)$ for all $t\in[0,1]$, and the subset of isometries of $\mbox{Isom}^o(\widetilde{AdS}^{n,1})$ which fix a given point is connected. Its subset $\mathcal{C}_0$ is:
\\- \emph{not empty:} since $\Omega$ is strictly convex and contains $0,$ for $x_0=(0,\varphi(0)),$ the isometry $f=id-\varphi(0)e_{n+1}$ is such that $f(x_0)=(0,0)$ and belongs to $\mathcal{C}_0$; 
\\- \emph{open in $\mathcal{C}$:} the strict convexity of $\Omega_f$ is an open condition; 
\\- \emph{closed in $\mathcal{C}$:} let us note that if $\overline{\Omega}$ is convex and $0\in\overline{\Omega},$ and if $n'\in\R^n$ stands for the unit vector field normal to $\partial\Omega$ and inward-directed, the inequality
$$II_{\partial\Omega}^{H}(X,X)\geq \langle II_{\partial M}(F_*X,F_*X),n'\rangle$$
holds for all $X\in T\partial\Omega$, where $II_{\partial M}$ is the second fundamental form of $\partial M$ in $\widetilde{AdS}^{n,1}$ and $F:\partial\Omega\rightarrow\partial M$ is the natural parametrization; this is proved in Lemma \ref{app lem ff} in the appendix. That proves that $\overline{\Omega}$ is in fact strictly convex (by the assumption (H2)), and implies that $f\in\overline{\mathcal{C}_0}$ in fact belongs to $\mathcal{C}_0$, that is that $\mathcal{C}_0$ is closed in $\mathcal{C}$.

So we have $\mathcal{C}_0=\mathcal{C},$ which ends the proof of the lemma. 
\end{proof}
Recalling (\ref{cond f chart adapted boundary point}), we see that the set of charts adapted to some boundary-point of $M_{\varphi}$ is a bounded subset of isometries of $B^n\times\R,$ so that
$$\mathcal{K}:=\{f(M_{\varphi}),\ f \mbox{ is a chart adapted to some boundary-point of }M_{\varphi}\}$$
belongs to a compact subset of $B^n\times\R;$ we deduce that there exists a controlled constant $r\in (0,1)$ such that 
$$\mathcal{K}\subset \overline{B}(0,r)\times [-\pi/2+\delta,\pi/2-\delta]$$
with $\delta=\pi/2-2\arctan(r);$ see Lemma \ref{app lem estim u} in the appendix for the expression of $\delta$.

Let us fix $x_0\in\partial M_{\varphi},$ and, using a chart adapted to $x_0$, assume that $x_0=0\in B^n$, $T_{x_0}M_{\varphi}=\R^n,$ $T_{x_0}\partial M_{\varphi}=\R^{n-1}\times\{0\}$ and $e_n$ is normal to $\partial M_{\varphi}$ and tangent to $M_{\varphi}$ at $x_0,$ that is, $M_{\varphi}$ is the graph of $\varphi:\overline{\Omega}\subset B^n\rightarrow\R$ with $0\in\partial\Omega,$ $T_0(\partial\Omega)=\{x_n=0\}$, $\varphi(0)=0$ and $\partial_i\varphi(0)=0$ for $i=1,\ldots,n.$ By Lemma \ref{est C1 lem Omega convex}, $\overline{\Omega}$ is strictly convex with respect to the hyperbolic metric; it moreover belongs to a ball $\overline{B}(0,r)$ for some controlled constant $r\in (0,1)$. Let us consider a function $\psi: B^n\rightarrow\R$ whose graph is an equidistant hypersurface of curvature $\mathcal{H}_m[\psi]=c_0\geq\sup_{\overline{\Omega}}H$ and such that
$$\psi(0)=\varphi(0)=0,\ \partial_i\psi(0)=\partial_i\varphi(0)=0,\ i=1,\ldots,n-1\ \mbox{and}\ \partial_n\psi(0)<0.$$
An equidistant hypersurface is an hypersurface at a fixed lorentzian distance to some totally geodesic hypersurface of $\widetilde{AdS}^{n,1}.$ It is umbilical, and has therefore constant curvature $\mathcal{H}_m$ (which may be chosen arbitrarily in $(0,+\infty)$). See Appendix \ref{appendix equidistant} for an explicit description of these hypersurfaces. Let us consider the operator
\begin{equation}\label{def L C1 bord}
Lu(X,Y)=\nabla^{H}du(X,Y)+\frac{1}{\mu}X.\mu\ Y.u+\frac{1}{\mu}Y.\mu\ X.u-\mu d\mu(\nabla^{H}u)du(X)du(Y);
\end{equation}
it is such that
\begin{equation}\label{rel L h C1 bord}
(Lu)_{ij}=\mu\sqrt{1-|\nabla^S u|^2}(\delta_{ik}-u_iu_k)h^k_j
\end{equation}
where $(h^k_j)_{kj}$ is the shape operator of $\mbox{graph}(u)$. This is a consequence of (\ref{intro def aij xpq}) and (\ref{intro def Su}), together with the relation $g_H=\mu^2g_S$ which implies
$$\nabla^Hdu(X,Y)=\nabla^S du(X,Y)-\frac{1}{\mu}\left(d\mu(X)du(Y)+d\mu(Y)du(X)-g_S(X,Y)du(\nabla^S\mu)\right)$$
with $\nabla^H\mu=1/\mu^2\ \nabla^S\mu.$
We have the following properties:
\begin{enumerate}
\item $L\varphi$ is positive definite on $\overline{\Omega}$ by (\ref{rel L h C1 bord}) since $\mbox{graph}(\varphi)$ is strictly convex with respect to the metric of $\widetilde{AdS}^{n,1},$ i.e. $L\varphi\geq c\ g_S$ for some controlled constant $c>0;$
\item $(L\psi)_{ij}=c_0^{1/m}\mu\sqrt{1-|\nabla^S \psi|^2}(\delta_{ij}-\psi_i\psi_j)$ since the graph of $\psi$ is an equidistant hypersurface of curvature $\mathcal{H}_m[\psi]=c_0$ and is therefore umbilical with $h^k_j=c_0^{1/m}\delta^k_j$.
\end{enumerate}
So, in the sense of quadratic forms, we have
\begin{equation}\label{C1 bord L psi phi}
L\psi< L\varphi\hspace{.3cm}\mbox{on}\hspace{.2cm}\overline{\Omega}
\end{equation}
if $|\nabla^S\psi|$ is chosen close to 1 on $\overline{B}(0,r),$ which may be obtained by Lemma \ref{app norm nabla u tends 1} in the appendix (taking $|\beta|$ large in that lemma, with $\beta<0$ so that $\partial_n\psi(0)<0$). Let us show that $\psi\leq\varphi$ on $\partial\Omega,$ so that $\psi$ is a lower barrier for the Dirichlet problem. We fix another point $x_1\in\partial\Omega$ and we want to show that $\psi(x_1)\leq\varphi(x_1).$ We assume by contradiction that $\psi(x_1)>\varphi(x_1)$. We consider a geodesic $\gamma:[0,t_1]\rightarrow\overline{\Omega}$ of $B^n$ such that $\gamma(0)=x_0=0$ and $\gamma(t_1)=x_1$ (geodesic with respect to the hyperbolic metric), and the function $f(t)=\varphi(\gamma(t))-\psi(\gamma(t)).$ We have $f(0)=0,$ $f'(0)>0$ and $f(t_1)<0,$ so that $f$ admits a positive interior maximum at some $t\in (0,t_1).$ We have $f'(t)=0$ and $f''(t)\leq 0,$ which yield, setting $\gamma_t=\gamma(t)$ and $\gamma'_t=\gamma'(t),$ 
\begin{equation}\label{C1 bord extr f}
d\varphi_{\gamma_t}(\gamma'_t)=d\psi_{\gamma_t}(\gamma'_t)
\end{equation}
and
\begin{equation}\label{C1 bord max f}
0\geq \nabla^H d\varphi_{\gamma_t}(\gamma'_t,\gamma'_t)- \nabla^H d\psi_{\gamma_t}(\gamma'_t,\gamma'_t)=L\varphi_{\gamma_t} (\gamma'_t,\gamma'_t)-L\psi_{\gamma_t} (\gamma'_t,\gamma'_t).
\end{equation}
For the last equality we used (\ref{def L C1 bord}), (\ref{C1 bord extr f}) and the fact that $\gamma'_t=\lambda_t\gamma_t$ for some $\lambda_t>0$ (geodesics from $x_0=0$ are pieces of straight lines) which implies
$$\nabla^H\mu(\gamma_t)=\gamma_t=\frac{1}{\lambda_t}\gamma'_t$$
(for all $x\in B^n,$ $\nabla^H\mu(x)=x,$ by a direct computation) and
\begin{eqnarray*}
d\mu_{\gamma_t}(\nabla^H\varphi(\gamma_t))=d\varphi_{\gamma_t}(\nabla^H\mu(\gamma_t))&=&\frac{1}{\lambda_t}d\varphi_{\gamma_t}(\gamma'_t)\\&=&\frac{1}{\lambda_t}d\psi_{\gamma_t}(\gamma'_t)=d\psi_{\gamma_t}(\nabla^H\mu(\gamma_t))=d\mu_{\gamma_t}(\nabla^H\psi(\gamma_t)).
\end{eqnarray*}
Inequality (\ref{C1 bord max f}) is a contradiction with (\ref{C1 bord L psi phi}).

We now conclude the estimate: since $\psi\leq\varphi$ on $\partial\Omega$ and $\mathcal{H}_m[\psi]=c_0\geq H$ in $\Omega,$ with $\psi$ admissible, every solution $u$ of the Dirichlet problem (\ref{Dirichlet problem}) satisfies $u\geq\psi$ on $\overline{\Omega}.$ Since $u(0)=\varphi(0)=\psi(0)$ we conclude that $u_n(0)\geq\psi_n(0),$ which yields the lower estimate on the boundary for the gradient of a solution. The upper estimate is straightforward: by the Maclaurin's inequality we have $\mathcal{H}_1[u]\geq c$ for some controlled positive constant $c$; if $\psi'$ is a spacelike function such that $\mathcal{H}_1[\psi']=c$ in $\Omega$ and $\psi'=\varphi$ on $\partial\Omega$ (obtained in \cite{Bar}), we have $\psi'\geq u$ on $\overline{\Omega}$, which yields the upper estimate of the gradient of $u$ on $\partial\Omega.$ That concludes the gradient estimate on the boundary, and therefore the $C^1$ estimate of a solution.

\section{The $C^2$ estimate}\label{section C2}
The aim of the section is to obtain the second estimate in (\ref{est C0 C2 meth res}). We first apply a maximum principle to reduce the estimate to an estimate of the second derivatives on the boundary, and then obtain the boundary estimates using barriers. The maximum principle for the second derivatives relies on the estimate of Urbas \cite{Ur}, and the estimates of the second derivatives on the boundary are adaptations of estimates of Ivochkina \cite{Iv1,Iv2} and Trudinger \cite{Tr95}, as in \cite{Bay1}. 

\subsection{Reduction to the $C^2$ estimate on the boundary} We follow Urbas \cite{Ur} very closely and do computations on the hypersurface $M.$ To stick to the notation used in \cite{Ur}, we write the equation of prescribed curvature in the form $\widetilde{F}=\psi$ with $\widetilde{F}=F^{\frac{1}{m}}$ and $\psi=H^{\frac{1}{m}},$ and take $m=2.$ Let us consider, for $x\in M$ and $\xi\in T_xM$ with $|\xi|=1,$
$$\widetilde{W}(x,\xi):=\eta(x)^\beta h_{ij}\xi^i\xi^j$$
where $\eta:M\rightarrow\R$ is a positive function and $\beta$ is a positive parameter that we will determine later. The function $\widetilde{W}$ attains its maximum at $x_0\in M,$ in the unit direction $\xi_0\in T_{x_0}M.$ We suppose that $x_0$ is an interior point of $M.$ Let us consider orthonormal vector fields $\widehat{e_1},\ldots,\widehat{e_n}\in\Gamma(TM)$ in a neighbourhood of $x_0$ such that $\nabla\widehat{e_j}(x_0)=0$ and $\widehat{e_1},\ldots,\widehat{e_n}$ is a basis of principal directions of $M$ at $x_0.$ We moreover assume that $\widehat{e_1}(x_0)=\xi_0.$ Note that $(h_{ij})_{ij}$ is diagonal at $x_0$, and assume that its eigenvalues satisfy $\lambda_1\geq\cdots\geq\lambda_n.$ Taking $\zeta=\widehat{e_1}$ the function $W(x)=\eta(x)^\beta h_{ab}\zeta^a\zeta^b(x)$ has an interior maximum at $x_0.$ Since the function $Z=h_{ab}\zeta^a\zeta^b$ is such that $Z_i=h_{11;i}$ and $Z_{ij}=h_{11;ij}$ at $x_0,$ the functions $Z$ and $h_{11}$ satisfy the same first and second order equations at that point. Therefore we have at $x_0$
\begin{equation}\label{ppemax C2 cond extr}
\frac{W_i}{W}=\beta\frac{\eta_i}{\eta}+\frac{h_{11;i}}{h_{11}} 
\end{equation}
for all $i=1,\ldots,n$, and, since $0\geq \widetilde{F}^{ij}(\log W)_{ij},$
\begin{equation}\label{ppemax C2 cond max}
0\geq \beta \widetilde{F}^{ij}\left(\frac{\eta_{ij}}{\eta}-\frac{\eta_i\eta_j}{\eta^2}\right)+\frac{1}{h_{11}}\widetilde{F}^{ij}h_{11;ij}-\frac{1}{h_{11}^2}\widetilde{F}^{ij}h_{11;i}h_{11;j}.
\end{equation}
The equation of prescribed curvature $\widetilde{F}=\psi$ differentiated twice reads
\begin{equation}\label{C2 max eq der 2}
\widetilde{F}^{kl,pq}h_{kl;i}h_{pq;j}+\widetilde{F}^{kl}h_{kl;ij}=\psi_{ij}.
\end{equation}
We will need the Ricci formula to exchange the indices of the second covariant derivative of the second fundamental form. It reads 
$$\nabla^2_{X,Y}h(Z,T)-\nabla^2_{Y,X}h(Z,T)=-h(R(X,Y)Z,T)-h(Z,R(X,Y)T)$$
where $R$ is the curvature tensor of $M$ given by
$$R(X,Y,Z,T)=R^0(X,Y,Z,T)-h(X,T)h(Y,Z)+h(X,Z)h(Y,T);$$
here $R^0$ is the curvature tensor of $\widetilde{AdS}^{n,1}$, given in terms of the metric $g$ of $\widetilde{AdS}^{n,1}$ by
$$R^0(X,Y,Z,T)=-g(X,T)g(Y,Z)+g(X,Z)g(Y,T).$$
We deduce from that formula and the Codazzi equations ($h_{ij;k}$ is symmetric in all the indices) that
\begin{eqnarray*}
h_{kl;ij}&=&h_{ij;kl}-g_{ik}h_{jl}+g_{lk}h_{ij}+g_{jl}h_{ik}-g_{ij}h_{lk}\\
&&-h_{ik}h^m_lh_{mj}+h_{lk}h_i^mh_{mj}+h_{jl}h_i^mh_{mk}-h_{ij}h^m_lh_{mk}.
\end{eqnarray*}
We deduce from (\ref{C2 max eq der 2}) that
\begin{eqnarray*}
\widetilde{F}^{kl}h_{ij;kl}&=&\left(\widetilde{F}-\lambda_i\sum_k\widetilde{F}^{kk}\right)\delta_{ij}-\widetilde{F}h_i^mh_{mj}\\
&&+\widetilde{F}^{kl}h^m_lh_{mk}h_{ij}-\widetilde{F}^{kl,pq}h_{pq;i}h_{kl;j}+\psi_{ij}
\end{eqnarray*}
and use that equation to replace the term $\widetilde{F}^{ij}h_{11;ij}$ in (\ref{ppemax C2 cond max}) (relabeling indices). We estimate $\frac{1}{\lambda_1}\widetilde{F},$ $-(n-1)\sigma_1/\widetilde{F},$ $-\widetilde{F}\lambda_1$ and $\psi_{11}/\lambda_1$ below by $-C(1+\lambda_1),$ and obtain 
\begin{eqnarray}
0&\geq& \beta \widetilde{F}^{ij}\left(\frac{\eta_{ij}}{\eta}-\frac{\eta_i\eta_j}{\eta^2}\right)+\sum_k\widetilde{F}^{kk}\lambda_k^2-C(1+\lambda_1)\label{ineg fin ppe max C2}\\
&&-\frac{1}{\lambda_1}\widetilde{F}^{kl,pq}h_{pq;1}h_{kl;1}-\frac{1}{\lambda_1^2}\widetilde{F}^{ij}h_{11;i}h_{11;j}.\nonumber
\end{eqnarray}
Note that this is exactly the key-inequality (2.8) of \cite{Ur}. It is known that a spacelike function $u:B^n\rightarrow\R$ in $\widetilde{AdS}^{n,1}$ satisfies $\sup_{B^n}u-\inf_{B^n}u<\pi$. In the context of the Dirichlet problem, we claim that there exist two constants $\alpha,\beta\in\R$ such that the graph of every solution of (\ref{Dirichlet problem}) belongs to $\overline{\Omega}\times[\alpha,\beta]$ with $\beta-\alpha\leq diam_S(\Omega)$ (the diameter of $\Omega$ with respect to the metric $g_S$). Recalling (\ref{C0 u phi d}), setting
$$\alpha=\inf_{x\in\overline{\Omega}}\sup_{y\in\partial\Omega}\left(\varphi(y)-d_S(x,y)\right)\hspace{.3cm}\mbox{and}\hspace{.3cm}\beta=\sup_{x\in\overline{\Omega}}\inf_{y\in\partial\Omega}\left(\varphi(y)+d_S(x,y)\right)$$
we have $\alpha\leq u(x)\leq\beta.$ The inequality $\beta-\alpha\leq diam_S(\Omega)$ will be obtained if we prove that for all $x,x'\in\overline{\Omega}$
$$\inf_{y\in\partial\Omega}\left(\varphi(y)+d_S(x,y)\right)-\sup_{y\in\partial\Omega}\left(\varphi(y)-d_S(x',y)\right)\leq diam_S(\Omega).$$
Since the term on the left-hand side is smaller than $d_S(x,y)+d_S(x',y)$ for all $y\in\partial\Omega,$ the inequality will be proved if we show that for all $x,x'\in\overline{\Omega}$ there exists $y\in\partial\Omega$ such that 
\begin{equation}\label{dem existence a ineq dS diam}
d_S(x,y)+d_S(x',y)\leq diam_S(\Omega).
\end{equation}
Let us consider a maximal geodesic $\gamma:I\rightarrow B^n$ of $(B^n,g_S)$ containing $x$ and $x'$ and such that $g_S(\gamma',\gamma')=1,$ and the points $y,y'\in\partial\Omega\cap\gamma$ such that, along $\gamma,$ $y\leq x\leq x'\leq y'$. We may assume that $d_S(x,y)=\min\left(d_S(x,y),d_S(x',y')\right).$ We obtain
\begin{eqnarray*}
d_S(x,y)+d_S(x',y)&=& d_S(x,y)+d_S(x',x)+d_S(x,y)\\
&\leq&d_S(x,y)+d_S(x',x)+d_S(x',y')=d_S(y,y')\leq diam_S(\Omega),
\end{eqnarray*}
which implies (\ref{dem existence a ineq dS diam}) and the existence of $\alpha$ and $\beta.$ So, taking $\mathcal{K}=\overline{\Omega}\times[\alpha,\beta]$ and $a\in\R$ such that $[\alpha,\beta]\subset (a,a+\pi)$, the graph of every solution of the Dirichlet problem (\ref{Dirichlet problem}) belongs to the compact subset  $\mathcal{K}\subset B^n\times (a,a+\pi)$. Let us consider now a linear form $L:\R^{n,2}\rightarrow\R$ and its restriction $\eta=L_{|\mathbb{H}^{n,1}}:\mathbb{H}^{n,1}\rightarrow\R$ to the quadric model (\ref{quadric model}) of Anti-de Sitter geometry: it is a solution of
\begin{equation}\label{est C2 ppe max eqn eta}
\widetilde{\nabla}d\eta=\eta g
\end{equation}
on $\mathbb{H}^{n,1}.$ Taking $L=X_{n+2}$ and composing $\eta=L_{|\mathbb{H}^{n,1}}$ by the covering map $\widetilde{AdS}^{n,1}\rightarrow\mathbb{H}^{n,1}$ we obtain the function $\mu\sin t$ on $\widetilde{AdS}^{n,1}=B^n\times\R.$ Since the vertical translation $t\mapsto t-a$ is an isometry of $\widetilde{AdS}^{n,1},$ the function $\eta=\mu\sin(t-a)$ is also a solution of (\ref{est C2 ppe max eqn eta}) on $\widetilde{AdS}^{n,1}$ and is moreover strictly positive on $B^n\times(a,a+\pi).$ Especially, there exists a positive constant $c_\mathcal{K},$ which only depends on the compact set $\mathcal{K}$ introduced above, such that $\eta\geq c_\mathcal{K}$ on $\mathcal{K}.$ Since $g\geq 0$ on spacelike graphs, we deduce that $\widetilde{\nabla}d\eta\geq c_\mathcal{K}\ g$ on the tangent vectors of $M$. 
\begin{lem}
We have the estimate
\begin{equation}\label{min Fijetaij ppe max C2}
\widetilde{F}^{ij}\eta_{ij}\geq c_\mathcal{K}\tau-c
\end{equation}
where $\tau=\sum_{i}\widetilde{F}^{ii}.$
\end{lem}
\begin{proof}
Since $\eta_{ij}=\widetilde{\nabla}d\eta(\widehat{e_i},\widehat{e_j})+h_{ij}d\eta(N)$ we have $\widetilde{F}^{ij}\eta_{ij}=\widetilde{F}^{ij}\widetilde{\nabla}d\eta(\widehat{e_i},\widehat{e_j})+\widetilde{F}d\eta(N).$ We estimate $\widetilde{F}d\eta(N)\geq -c$ and
$$\widetilde{F}^{ij}\widetilde{\nabla}d\eta(\widehat{e_i},\widehat{e_j})=\sum_i\widetilde{F}^{ii}\widetilde{\nabla}d\eta(\widehat{e_i},\widehat{e_i})\geq c_\mathcal{K}\sum_i\widetilde{F}^{ii}g(\widehat{e_i},\widehat{e_i})=c_\mathcal{K}\tau.$$
\end{proof}
The inequalities (\ref{ineg fin ppe max C2}) and (\ref{min Fijetaij ppe max C2}) are exactly the inequalities (2.8) and (2.12) obtained by Urbas in \cite{Ur}. We may then obtain a bound of $\eta^\beta\lambda_1$ at $x_0$ following \cite{Ur} without any modification, if $\beta$ is chosen sufficiently large, which concludes the estimate.

\subsection{The $C^2$ estimate on the boundary}
 If $e_1^0,\ldots,e_n^0$ stands for the canonical basis of $\R^n,$ the frame $e_i:=\frac{1}{\lambda} e_i^0,$ $i=1,\ldots,n$ is orthonormal with respect to $g_S.$ Let us fix $x_0\in\partial\Omega$ and assume that $e_1,\ldots,e_{n-1}$ are tangent to $\partial\Omega$ at $x_0$ and that $e_n$ is the inner normal at that point. In that section we denote by $u_i,$ $u_{ij},$ ... the components in $e_i,\ i=1,\ldots n$ of the covariant derivatives of $u$ in $(B^n,g_S)$. The purpose of the section is to obtain a bound
 $$|u_{ij}(x_0)|\leq C,\ \ \forall i,j=1,\ldots n,$$
 for a constant $C$ which only depends on the Dirichlet data and on constants $M\geq 0$ and $\theta\in (0,1]$ such that $\sup_{\overline{\Omega}}|u|\leq M$ and $\sup_{\overline{\Omega}}|\nabla^S u|\leq 1-\theta.$ Let us note that the obtention of bounds on $u_{ij}(x_0),$ $1\leq i,j\leq n-1$ is straightforward. So we focus on the obtention of the estimates of the mixed second derivatives $u_{in}(x_0),$ $1\leq i\leq n-1$ and the double normal derivatives $u_{nn}(x_0).$ The following inequality due to Ivochkina, Lin and Trudinger \cite[ineq. (26)]{LT} will be essential for the estimates: there exists a constant $C_0=C_0(n,m)$ such that in $\Gamma_m$, and for all $k=1,\ldots,n,$
 \begin{equation}\label{est C2b ineg ILT1}
 \sigma_{m-1,k}\lambda_k^2\leq\lambda_k\sigma_m+C_0\sum_{i\neq k}\sigma_{m-1,i}\lambda_i^2.
 \end{equation}
 That inequality also takes the useful form
 \begin{equation}\label{est C2b ineg ILT2}
 \sigma_{m+1,k}\leq C_0\sum_{i\neq k}\sigma_{m-1,i}\lambda_i^2.
 \end{equation}
 Let us first introduce some notation. We consider the positive definite symmetric matrix 
\begin{equation}\label{est C2b def matrix P}
P(p)=\left(\gamma_{ij}\right)_{i,j}=\left(\delta_{ij}+\frac{p_ip_j}{w(1+w)}\right)_{i,j},
\end{equation}
where $w=\sqrt{1-|p|^2}.$ It is such that $P(p)^2=A(p)^{-1}$ and setting 
\begin{equation}\label{def bij xpq}
(b_{ij}(x,p,q))_{ij}:=P(p)\ q\  P(p)+d(\log\mu)_x(\sum_sp_s e_s)\ I
\end{equation}
and
\begin{equation}\label{def Fmxpq estC2b}
\mathcal{G}_m(x,p,q):=F_m((b_{ij}(x,p,q))_{ij}),
\end{equation}
the $m^{th}$ symmetric function of the eigenvalues of $(b_{ij}(x,p,q))_{ij},$ the equation of prescribed curvature $\mathcal{H}_m[u]=H$ reads
\begin{equation}\label{estC2b eqn Gm=f}
\mathcal{G}_m(x,\nabla^S u,\nabla^Sdu)=f(x,\nabla^S u)
\end{equation}
with
$$f(x,p)=\frac{n!}{m!(n-m)!}\mu^m(x)(1-|p|^2)^{\frac{m}{2}}H(x).$$
Here and below we use the notation $p=\nabla^S u\in B(0,1)\subset\R^n$ and $q=\nabla^Sd u\in S_n(\R)$ to denote the components of $\nabla^S u$ and $\nabla^Sdu$ in the orthonormal basis $(e_i)_{1\leq i\leq n}.$ We will need the following formulas (see \cite{GS1} for similar formulas in $\mathbb{H}^n$):
\begin{lem}\label{lem expr partial bij}
The partial derivatives of $b_{ij}$ are given by
$$\frac{1}{\lambda}\frac{\partial b_{ij}}{\partial x_t}(x,p,q)=\left(\nabla^S d(\log\mu)(e_t,\sum_sp_se_s)+d(\log\mu)(\sum_sp_s\nabla^S_{e_t}e_s)\right)\delta_{ij},$$
\begin{eqnarray*}
\frac{\partial b_{ij}}{\partial p_s}(x,p,q)&=&\frac{1}{1+w}\sum_{k,l,r}q_{kl}\left(p_r\gamma_{is}+\frac{p_i}{w}\gamma_{rs}\right)\gamma_{rk}\gamma_{lj}\\
&&+\frac{1}{1+w}\sum_{k,l,r}q_{kl}\left(p_r\gamma_{js}+\frac{p_j}{w}\gamma_{rs}\right)\gamma_{rk}\gamma_{li}+d(\log\mu)(e_s)\ \delta_{ij}
\end{eqnarray*}
and 
$$\frac{\partial b_{ij}}{\partial q_{kl}}(x,p,q)=\gamma_{ik}\gamma_{lj}.$$
\end{lem}
\begin{proof}
The first and the last formulas are straightforward. Let us mention the main steps for the obtention of $\frac{\partial b_{ij}}{\partial p_s}$. By (\ref{def bij xpq}) we have $b_{ij}=\sum_{k,l}\gamma_{ik}q_{kl}\gamma_{lj}+d(\log\mu)_x(\sum_sp_se_s)\ \delta_{ij}$ and
\begin{equation}\label{partiel bij function q mu} 
\frac{\partial b_{ij}}{\partial p_s}=\sum_{k,l}\left(\partial_{p_s}\gamma_{ik}\ q_{kl}\gamma_{lj}+\gamma_{ik}q_{kl}\ \partial_{p_s}\gamma_{lj}\right)+d(\log\mu)_x(e_s)\ \delta_{ij}.
\end{equation}
Introducing the inverse of the matrix $(\gamma_{ij})_{ij}$, 
\begin{equation*}
(\gamma'_{ij})_{ij}:=\left(\delta_{ij}-\frac{1}{1+w}p_ip_j\right)_{ij},
\end{equation*}
a direct computation using $\partial_{p_s}w=-p_s/w$ and 
\begin{equation}\label{sum k gammaik pk}
\sum_k\gamma_{ik}\ p_k=\frac{p_i}{w}
\end{equation}
yields
$$\sum_k\partial_{p_s}\gamma_{ik}\ \gamma'_{kj}=-\sum_k\gamma_{ik}\ \partial_{p_s}\gamma'_{kj}=\frac{1}{1+w}\left(p_j\gamma_{is}+\frac{p_i}{w}\gamma_{js}\right),$$
which implies
\begin{equation}\label{partiel gamma ij function q mu} 
\partial_{p_s}\gamma_{ik}=\frac{1}{1+w}\sum_r\left(p_r\gamma_{is}+\frac{p_i}{w}\gamma_{rs}\right)\gamma_{rk}.
\end{equation}
The expression of $\frac{\partial b_{ij}}{\partial p_s}$ follows from (\ref{partiel bij function q mu}) and (\ref{partiel gamma ij function q mu}).
\end{proof}
By Lemma \ref{lem expr partial bij}, setting $(f)_t:=\frac{1}{\lambda}\frac{\partial}{\partial x_t}\left(f(x,\nabla^Su(x))\right),$ the derivative of the equation of prescribed curvature (\ref{estC2b eqn Gm=f}) reads
\begin{eqnarray}\label{eq presc curv derivee interm}
(f)_t&=&
\frac{1}{\lambda}\frac{\partial\mathcal{G}_m}{\partial x_t}[u]+\sum_s\frac{\partial\mathcal{G}_m}{\partial p_s}[u]\left(u_{st}+du(\nabla^S_{e_t}e_s)\right)\nonumber\\
&&+\sum_{k,l}\frac{\partial\mathcal{G}_m}{\partial q_{kl}}[u] \left(u_{kl;t}+\nabla^S du(\nabla^S_{e_t}e_k,e_l)+\nabla^S du(e_k,\nabla^S_{e_t}e_l)\right)\label{eq presc curv derivee interm}
\end{eqnarray}
with
\begin{eqnarray*}
\frac{1}{\lambda}\frac{\partial\mathcal{G}_m}{\partial x_t}[u]&=&\frac{1}{\lambda}\sum_{i,j}F^{ij}\frac{\partial b_{ij}}{\partial x_t}(x,\nabla^S u,\nabla^S du)\\
&=&\left(\nabla^S d(\log\mu)(e_t,\nabla^S u)+d(\log\mu)(\sum_su_s\nabla^S_{e_t}e_s)\right)\sum_{i}F^{ii},
\end{eqnarray*}
\begin{eqnarray}
\frac{\partial\mathcal{G}_m}{\partial p_s}[u]&=&\sum_{i,j}F^{ij}\frac{\partial b_{ij}}{\partial p_s}(x,\nabla^S u,\nabla^S du)\nonumber\\
&=&\frac{2}{1+w}\sum_{i,j,k,l,r}F^{ij}u_{kl}\left(u_r\gamma_{is}+\frac{u_i}{w}\gamma_{rs}\right)\gamma_{rk}\gamma_{lj}+d(\log \mu)(e_s)\sum_{i}F^{ii}\label{est C2b Gm ps}
\end{eqnarray}
and 
\begin{equation}\label{est C2b Gm qkl Fii}
\frac{\partial\mathcal{G}_m}{\partial q_{kl}}[u] =\sum_{i,j}F^{ij}\frac{\partial b_{ij}}{\partial q_{kl}}(x,\nabla^S u,\nabla^S du)=\sum_{i,j}F^{ij}\gamma_{ik}\gamma_{lj}.
\end{equation}

Let $\tau_{\alpha},$ $\alpha=1,\ldots,n$ be a basis of $\R^n$ which induces by $x\mapsto (x,u(x))$ a basis of principal directions of $M=\mbox{graph}(u)$ in $\widetilde{AdS}^{n,1}$ which is orthonormal with respect to the metric $g'=g_{S}-dt^2$. Let us denote $u_\alpha=du(\tau_\alpha),$ $u_{\alpha\beta}=\nabla^S d u(\tau_\alpha,\tau_\beta),$ $u_{\alpha i}=\nabla^S d u(\tau_\alpha,e_i),$ etc... and by $\gamma_{\alpha},$ $\alpha=1,\ldots, n$ the eigenvalues of the matrix $(b_{ij}(x,\nabla^Su,\nabla^Sdu))_{ij}$. All the symmetric functions below will be symmetric functions of the $\gamma_{\alpha},$ $\alpha=1,\ldots, n$. Since the basis $\tau_{\alpha}$ induces an orthonormal basis of $TM$ for the metric $g',$ we have $g'_{\alpha\beta}=\delta_{\alpha\beta}.$ Moreover, since 
$$g'^{\alpha\beta}u_{\beta\gamma}+d(\log\mu)(\nabla^S u)\delta_{\gamma}^\alpha=\gamma_{\alpha}\delta^{\alpha}_{\gamma}$$
by (\ref{intro def aij xpq})-(\ref{intro def Su}), we deduce that the matrix $(u_{\alpha\beta})_{\alpha,\beta}$ is diagonal, given by
\begin{equation}\label{uab diagonal}
(u_{\alpha\beta})_{\alpha,\beta}= \mbox{diag}\left(\gamma_{\alpha}-d(\log\mu)(\nabla^S u),\ \alpha=1,\ldots,n\right).
\end{equation}
By using the basis $\tau_{\alpha},$ $\alpha=1,\ldots,n$ and after some computations (details will be given below) we obtain from the expressions above 
\begin{equation}\label{partiel Gm ps interm}
\sum_s\frac{\partial\mathcal{G}_m}{\partial p_s}[u]u_{st}=2\sum_{\alpha}\sigma_{m-1,\alpha}u_\alpha u_{\alpha\alpha} u_{\alpha t}+(\sum_iF^{ii})\nabla^Sdu(\nabla^S(\log \mu),e_t)
\end{equation}
and
\begin{equation}
\sum_{k,l}\frac{\partial\mathcal{G}_m}{\partial q_{kl}}[u] u_{kl;t}=\sum_{\alpha}\sigma_{m-1,\alpha}u_{\alpha\alpha; t},
\end{equation}
so that (\ref{eq presc curv derivee interm}) may be written in the form
\begin{equation}\label{eq presc curv derivee C2b}
(f)_t=\sum_{\alpha=1}^n\sigma_{m-1,\alpha}(u_{\alpha\alpha;t}+2u_\alpha u_{\alpha\alpha}u_{t\alpha})+(\sum_i F^{ii})\nabla^S du(\nabla^S(\log\mu),e_t)+\mathcal{T}
\end{equation}
with 
\begin{equation}\label{formule T est C2b}
|\mathcal{T}|\leq C(\sigma_{m-1}+\sum_{\alpha}\sigma_{m-1,\alpha}|u_{\alpha\alpha}|).
\end{equation}
We have also used the estimates 
\begin{equation}\label{estim C2b terms in T}
\frac{1}{\lambda}\left|\frac{\partial\mathcal{G}_m}{\partial x_t}[u]\right|\leq \sigma_{m-1},\hspace{.5cm}\left|\frac{\partial\mathcal{G}_m}{\partial p_s}[u]\right|\leq C\left(\sigma_{m-1}+\sum_{\alpha}\sigma_{m-1,\alpha}|u_{\alpha\alpha}|\right)
\end{equation}
and, since
\begin{equation}\label{est C2b partial Gm qjk tau}
\frac{\partial\mathcal{G}_m}{\partial q_{jk}}[u]=\sum_{\alpha}\tau^j_{\alpha}\sigma_{m-1,\alpha}\tau^k_{\alpha},
\end{equation}
that
$$\sum_{j,k}\frac{\partial\mathcal{G}_m}{\partial q_{jk}}[u]\ \nabla^Sdu(\nabla^S_{e_t}e_j,e_k)=\sum_{\alpha}\sigma_{m-1,\alpha}\ \nabla^Sdu(\sum_j\tau^j_\alpha\nabla^S_{e_t}e_j,\tau_{\alpha})$$
is bounded by $C\sum_{\alpha}\sigma_{m-1,\alpha}|u_{\alpha\alpha}|$ to see that the remaining terms of (\ref{eq presc curv derivee interm}) are included in $\mathcal{T}$ satisfying (\ref{formule T est C2b}). Let us briefly explain how to obtain (\ref{partiel Gm ps interm}) ((\ref{estim C2b terms in T}) is obtained similarly). Setting
$$\mathcal{A}=\sum_s\left(\sum_{i,j,k,l,r}F^{ij}u_{kl}u_r\gamma_{is}\gamma_{rk}\gamma_{lj}\right)u_{st}\hspace{.3cm}\mbox{and}\hspace{.3cm}\mathcal{B}=\sum_s\left(\sum_{i,j,k,l,r}F^{ij}u_{kl}u_i\gamma_{rs}\gamma_{rk}\gamma_{lj}\right)u_{st}$$
(\ref{est C2b Gm ps}) implies that
$$\sum_s\frac{\partial\mathcal{G}_m}{\partial p_s}[u]u_{st}=\frac{2}{1+w}\left(\mathcal{A}+\frac{1}{w}\mathcal{B}\right)+(\sum_iF^{ii})\nabla^Sdu(\nabla^S(\log \mu),e_t).$$
$\mathcal{A}$ is computed as follows: if the coefficients $\tau^s_{\alpha}$ are such that $\tau_{\alpha}=\sum_s\tau^s_{\alpha}e_s,$ using that $\sum_{i,j}F^{ij}\gamma_{is}\gamma_{lj}=\sum_{\alpha}\sigma_{m-1,\alpha}\tau^s_{\alpha}\tau^l_{\alpha}$ (both terms are expressions of $\frac{\partial\mathcal{G}_m}{\partial q_{sl}}[u]$, using on the left hand-side $\mathcal{G}_m[u]=F((b_{ij})_{ij})$ together with the expression of $\frac{\partial b_{ij}}{\partial q_{sl}}$ given in Lemma \ref{lem expr partial bij}, and on the right-hand side (\ref{est C2b partial Gm qjk tau})) and since $\sum_ru_r\gamma_{rk}=\frac{u_k}{w}$ (by (\ref{sum k gammaik pk})) we have
$$\mathcal{A}=\frac{1}{w}\sum_{\alpha}\sigma_{m-1,\alpha}\sum_{s,k,l}u_{kl}\tau_{\alpha}^s\tau_{\alpha}^lu_ku_{st}=\frac{1}{w}\sum_{\alpha}\sigma_{m-1,\alpha}\left(\sum_ku_{k\alpha}u_k\right)u_{\alpha t}.$$
Now, writing $\nabla^Su=\sum_{\beta}u^{\beta}\tau_{\beta}$ we have $\sum_ku_{k\alpha}u_k=\nabla^Sdu(\nabla^Su,\tau_{\alpha})=u^{\alpha}u_{\alpha\alpha}.$ Since the basis $\tau_{\alpha},$ $\alpha=1,\ldots,n$ induces an orthonormal basis of $TM$ for the metric $g'=g_S-dt^2$ we have $g_S(\tau_{\alpha},\tau_{\beta})=u_{\alpha}u_{\beta}+\delta_{\alpha\beta}$ and
$$u_{\alpha}=\sum_{\beta}g_S(\tau_{\alpha},\tau_{\beta})\ u^{\beta}=\sum_{\beta}\left(u_{\alpha}u_{\beta}+\delta_{\alpha\beta}\right)u^{\beta}=u_{\alpha}|\nabla^Su|^2+u^{\alpha},$$
that is $u^{\alpha}=w^2u_{\alpha}.$ So $\sum_ku_{k\alpha}u_k=w^2u_{\alpha}u_{\alpha\alpha}$ and $\displaystyle{\mathcal{A}=w\sum_{\alpha}\sigma_{m-1,\alpha}u_{\alpha}u_{\alpha\alpha}u_{t\alpha}.}$ Let us now compute $\mathcal{B}$. Using that $u_i=w\sum_{r'}\gamma_{ir'}u_{r'}$ (by (\ref{sum k gammaik pk})) and since $\sum_{r}\gamma_{rs}\gamma_{rk}={g'}^{sk}$ we have 
$$\displaystyle{\mathcal{B}=w\sum_{\alpha}\sigma_{m-1,\alpha}u_{\alpha}\sum_{s,k}{g'}^{sk}u_{st}u_{k\alpha}}$$ 
with
$$\sum_{s,k}{g'}^{sk}u_{st}u_{k\alpha}=\sum_{\beta,\gamma}{g'}^{\beta\gamma}u_{\beta t}u_{\gamma\alpha}=\sum_{\beta}u_{\beta t}u_{\beta\alpha}=u_{\alpha t}u_{\alpha\alpha}$$
where we have used that ${g'}^{\beta\gamma}=\delta_{\beta\gamma}$ (the basis induced by $\tau_{\alpha}$ is orthonormal w.r.t. $g'$) and that $u_{\alpha\beta}$ is diagonal; we deduce that $\mathcal{B}=w\sum_{\alpha}\sigma_{m-1,\alpha}u_{\alpha}u_{\alpha\alpha}u_{t\alpha},$ and (\ref{partiel Gm ps interm}) follows.
\\

The estimates will rely on the following lemma:
\begin{lem}\label{lem g ineg FijWij}
Let us denote by $\overline{B}_r(x_0)$ the closed ball in $(B^n,g_S)$ with center $x_0$ and radius $r$, and let
\begin{eqnarray*}
g:\hspace{.5cm}\overline{\Omega}\cap \overline{B}_r(x_0)\times B(0,1)&\rightarrow&\R\\
(x,p)&\mapsto & g(x,p)
\end{eqnarray*}
be a function of class $C^2,$ concave with respect to $p,$ and
\begin{equation}\label{def W der sec bord}
W=g(.,\nabla^S u)-\frac{K}{2}\sum_{s=1}^{n-1}(u_s-u_s(x_0))^2.
\end{equation}
If $D_{r,\theta}$ stands for the compact $\overline{\Omega}\cap\overline{B}_r(x_0)\times \overline{B}(0,1-\theta),$ for $K=K(n,m,\theta,C_0,\|g\|_{1,D_{r,\theta}},\|\mu\|_{1,\overline{B}_r(x_0)})$ sufficiently large, $W$ satisfies in $\Omega\cap B_r(x_0)$ the inequality
\begin{equation}\label{ineg FijWij lem}
\sum_{i,j}\frac{\partial\mathcal{G}_m}{\partial q_{ij}}[u]\ W'_{ij}\leq C_1\left(1+|\nabla^S W|+\mathcal{G}_{m-1}[u]+\sum_{i,j}\frac{\partial\mathcal{G}_m}{\partial q_{ij}}[u] W_iW_j\right)
\end{equation}
where
\begin{equation}\label{est C2b def W'ij}
W'_{ij}:=W_{ij}+(\delta_{ij}-u_iu_j)\ dW(\nabla^S(\log\mu))
\end{equation}
and $C_1=C_1(n,m,\Omega,H,\theta,C_0,\|g\|_{2,D_{r,\theta}},\|\mu\|_{2,\overline{B}_r(x_0)}).$
\end{lem}
Let $\eta^{\alpha}_s,$ $\alpha,s=1,\ldots,n$ be the coefficients such that $e_s=\sum_{\alpha=1}^n\eta^{\alpha}_s\tau_\alpha,$ $s=1,\ldots,n.$ The following lemma will be crucial for the proof: 
\begin{lem}\label{lem delta deltap}
Let us fix $\varepsilon\in (0,1)$ and $\delta_{\varepsilon}=\frac{\varepsilon^2}{n-1}(1-(1-\theta)^2)$ where $\theta\in (0,1]$ is such that $\sup_{\overline{\Omega}}|\nabla^S u|\leq 1-\theta.$ If $\sum_{s=1}^{n-1}|\eta^1_s|^2<\delta_{\varepsilon}$ then, for $\alpha\geq 2,$ $\sum_{s=1}^{n-1}|\eta^\alpha_s|^2\geq\delta'_{\varepsilon}$ with
$$\delta'_{\varepsilon}=\frac{(1-\varepsilon)^2}{((n-2)!)^2(n-1)}(1-(1-\theta)^2).$$
\end{lem}
\noindent We refer to Lemma 4.3 in \cite{Bay1} for the proof, which is strictly analogous.
\\
\\\textit{Proof of Lemma \ref{lem g ineg FijWij}.} 
Inverting (\ref{est C2b Gm qkl Fii}) we have $F^{ij}=\sum_{k,l}\frac{\partial\mathcal{G}_m}{\partial q_{kl}}[u] {\gamma'}^{ik}{\gamma'}^{jl}$ where $({\gamma'}^{ij})_{ij}$ is the inverse of $(\gamma_{ij})_{ij},$ which implies that $\sum_{i}F^{ii}=\sum_{k,l}\frac{\partial\mathcal{G}_m}{\partial q_{kl}}[u] (\delta_{kl}-u_ku_l),$ and the definition (\ref{est C2b def W'ij}) of $W'_{ij}$ yields
\begin{eqnarray*}
\sum_{i,j=1}^n\frac{\partial \mathcal{G}_m}{\partial q_{ij}}[u]\ W'_{ij}&=&\sum_{\alpha}\sigma_{m-1,\alpha}W_{\alpha\alpha}+(\sum_{i}F^{ii})\ dW(\nabla^S(\log\mu)).
\end{eqnarray*}
We compute, for $\alpha=1,\ldots, n,$
\begin{equation}\label{expr Walpha}
W_{\alpha}=g_\alpha+\sum_{t=1}^ng_{p_t}\left(u_{t\alpha}+du(\nabla^S_{\tau_\alpha}e_t)\right)-K\sum_{s=1}^{n-1}\left(u_{s\alpha}+du(\nabla^S_{\tau_\alpha}e_s)\right)(u_s-u_s(x_0))
\end{equation}
and
\begin{eqnarray}
W_{\alpha\alpha}&=&g_{\alpha\alpha}+2\sum_{t=1}^ng_{\alpha p_t}\left(u_{t\alpha}+du(\nabla^S_{\tau_\alpha}e_t)\right)+\sum_{s,t=1}^ng_{p_t p_s}\left(u_{t\alpha}+du(\nabla^S_{\tau_\alpha}e_t)\right)\left(u_{s\alpha}+du(\nabla^S_{\tau_\alpha}e_s)\right)\nonumber\\
&&+\sum_{t=1}^ng_{p_t}\left(u_{t\alpha;\alpha}+2\nabla^S du(\nabla^S_{\tau_\alpha}e_t,\tau_{\alpha})+du({\nabla^{S}}^2_{\tau_\alpha,\tau_\alpha}e_t)\right)\nonumber\\
&&-K\sum_{s=1}^{n-1}\left(u_{s\alpha;\alpha}+2\nabla^S du(\nabla^S_{\tau_\alpha}e_s,\tau_{\alpha})+du({\nabla^S}^2_{\tau_\alpha,\tau_\alpha}e_s)\right)(u_s-u_s(x_0)))\nonumber\\
&&-K\sum_{s=1}^{n-1}\left(u_{s\alpha}+du(\nabla^S_{\tau_{\alpha}}e_s)\right)^2.\label{eqn Walphaalpha}
\end{eqnarray}
Since $g$ is concave with respect to $p,$ we have $\sum_{s,t}g_{p_tp_s}\eta^{\alpha}_s\eta^{\alpha}_t\leq 0.$ Using the Ricci formula $u_{t\alpha;\alpha}=u_{\alpha\alpha; t}+u_t{g_S}_{\alpha\alpha}-u_{\alpha}{g_S}_{\alpha t}$, a direct computation and straightforward estimates show that
\begin{eqnarray*}
\sum_{\alpha}\sigma_{m-1,\alpha}W_{\alpha\alpha}&\leq &\sum_{t=1}^n\sum_{\alpha}\sigma_{m-1,\alpha}g_{p_t}u_{\alpha\alpha;t}-K\sum_{s=1}^{n-1}\sum_{\alpha}\sigma_{m-1,\alpha}u_{\alpha\alpha;s}(u_s-u_s(x_0))\\
&&-K\sum_{\alpha}\sum_{s=1}^{n-1}\sigma_{m-1,\alpha}u_{s\alpha}^2+T_1
\end{eqnarray*}
where here and below $T_i$, for $i=1,2,3,\ldots$, denotes a sum of terms such that 
$$|T_i|\leq C_i\left(1+|\nabla^S W|+\sigma_{m-1}+\sum_{\alpha=1}^n\sigma_{m-1,\alpha}|u_{\alpha\alpha}|\right)$$
for some controlled constant $C_i$. Using the differentiated equation (\ref{eq presc curv derivee C2b}) to replace the third derivatives of $u$, and since
\begin{eqnarray*}
-2\sum_{\alpha}\sigma_{m-1,\alpha}u_{\alpha}u_{\alpha\alpha}W_{\alpha}&=&-2\sum_{\alpha}\sigma_{m-1,\alpha}u_{\alpha}u_{\alpha\alpha}\sum_{t=1}^ng_{p_t}u_{t\alpha}\\
&&+2K\sum_{\alpha}\sigma_{m-1,\alpha}u_{\alpha}u_{\alpha\alpha}\sum_{s=1}^{n-1}u_{s\alpha}(u_s-u_s(x_0))+T_2
\end{eqnarray*}
we obtain
\begin{eqnarray*}
\sum_{\alpha}\sigma_{m-1,\alpha}W_{\alpha\alpha}&\leq &\sum_{t=1}^ng_{p_t}(f)_t-K\sum_{s=1}^{n-1}(f)_s(u_s-u_s(x_0))\\
&&+\sum_{\alpha}\sigma_{m-1,\alpha}(-2u_{\alpha}u_{\alpha\alpha}W_{\alpha})-K\sum_{\alpha}\sigma_{m-1,\alpha}\sum_{s=1}^{n-1}|\eta_s^{\alpha}|^2u_{\alpha\alpha}^2\\
&&-(\sum_iF^{ii})\sum_{\alpha}(\log\mu)^{\alpha}\left(\sum_{t=1}^ng_{p_t}u_{\alpha t}-K\sum_{s=1}^{n-1}u_{s\alpha}(u_s-u_s(x_0))\right)+T_3.
\end{eqnarray*}
Here $(\log \mu)^{\alpha},$ $\alpha=1,\ldots,n,$ are the coefficients of $\nabla^S(\log\mu)$ in the basis $\tau_{\alpha}.$ We have the bound
\begin{equation}\label{pf ineq W eqn1}
\sum_{t=1}^ng_{p_t}(f)_t-K\sum_{s=1}^{n-1}(f)_s(u_s-u_s(x_0))\leq C(1+|\nabla^S W|).
\end{equation}
Indeed, since
$$(f)_t=\frac{1}{\lambda}\frac{\partial}{\partial x_t}(f(x,\nabla^S u(x)))=f_t+\sum_{r=1}^nu_{rt}f_{p_r}$$
and since there exists a controlled constant $C$ such that
$$\left|\sum_{t=1}^ng_{p_t}f_t-K\sum_{s=1}^{n-1}f_s(u_s-u_s(x_0))\right|\leq C,$$
we only have to bound the other contribution
\begin{eqnarray*}
\sum_{t=1}^ng_{p_t}\left(\sum_{r=1}^nu_{rt}f_{p_r}\right)-K\sum_{s=1}^{n-1}\left(\sum_{r=1}^nu_{rs}f_{p_r}\right)(u_s-u_s(x_0))&=&\sum_{r=1}^nf_{p_r}W_r+bounded\ terms\\
&\leq&C(1+|\nabla^S W|),
\end{eqnarray*}
so that (\ref{pf ineq W eqn1}) follows. Using finally that
$$\sum_{t=1}^ng_{p_t}u_{\alpha t}-K\sum_{s=1}^{n-1}u_{s\alpha}(u_s-u_s(x_0))=W_{\alpha}+bounded\ terms$$
we deduce that
\begin{eqnarray*}
\sum_{\alpha}\sigma_{m-1,\alpha}W_{\alpha\alpha}&\leq& \sum_{\alpha}\sigma_{m-1,\alpha}(-2u_{\alpha}u_{\alpha\alpha}W_{\alpha})-K\sum_{\alpha}\sigma_{m-1,\alpha}\sum_{s=1}^{n-1}|\eta_s^{\alpha}|^2u_{\alpha\alpha}^2\\
&&-(\sum_iF^{ii})\sum_{\alpha}(\log\mu)^{\alpha}W_{\alpha}+T_4
\end{eqnarray*}
which implies that
\begin{eqnarray}
\sum_{i,j=1}^n\frac{\partial \mathcal{G}_m}{\partial q_{ij}}[u]\ W'_{ij}&\leq&  C\left(1+|\nabla^S W|+\sigma_{m-1}+\sum_{\alpha=1}^n\sigma_{m-1,\alpha}|u_{\alpha\alpha}|\right)\nonumber\\
&&+\sum_{\alpha=1}^n\sigma_{m-1,\alpha}\left(-Ku_{\alpha \alpha}^2\sum_{s=1}^{n-1}|\eta^{\alpha}_s|^2-2u_{\alpha}u_{\alpha\alpha}W_{\alpha}\right).\label{sum Waa lem interm}
\end{eqnarray}
As in \cite{Bay1}, we now use the term $-K\sum_{\alpha=1}^n\sigma_{m-1,\alpha}u_{\alpha \alpha}^2\sum_{s=1}^{n-1}|\eta^{\alpha}_s|^2$ to balance the bad term $\sum_{\alpha=1}^n\sigma_{m-1,\alpha}u_{\alpha\alpha}^2$ which appears when bounding in (\ref{sum Waa lem interm}) the terms $\sum_{\alpha=1}^n\sigma_{m-1,\alpha}|u_{\alpha\alpha}|$ and $\sum_{\alpha=1}^n\sigma_{m-1,\alpha}|u_{\alpha}u_{\alpha\alpha}W_\alpha|.$ We distinguish two kinds of points: either $\sum_{s=1}^{n-1}|\eta^{\alpha}_s|^2\geq\delta_{\varepsilon}$ for all $\alpha\in\{1,2,\ldots,n\}$, or $\sum_{s=1}^{n-1}|\eta^{\alpha}_s|^2<\delta_{\varepsilon}$ for some $\alpha\in\{1,2,\ldots,n\}$, where $\delta_{\varepsilon}$ is introduced in Lemma \ref{lem delta deltap}. For the points of the first type, we bound 
$$C\sum_{\alpha=1}^n\sigma_{m-1,\alpha}|u_{\alpha\alpha}|\leq \frac{C'}{\delta_{\varepsilon}}\sigma_{m-1}+\frac{\delta_{\varepsilon}}{2}\sum_{\alpha=1}^n\sigma_{m-1,\alpha}u_{\alpha\alpha}^2$$
and 
$$2\sum_{\alpha=1}^n\sigma_{m-1,\alpha}|u_{\alpha}u_{\alpha\alpha}W_\alpha|\leq \frac{C''}{\delta_{\varepsilon}}\sum_{\alpha=1}^n\sigma_{m-1,\alpha}W_\alpha^2+\frac{\delta_{\varepsilon}}{2}\sum_{\alpha=1}^n\sigma_{m-1,\alpha}u_{\alpha\alpha}^2$$
and (\ref{ineg FijWij lem}) holds for $K\geq 1.$ For the points of the second type, assuming that $\alpha=1,$ Lemma \ref{lem delta deltap} implies that  $\sum_{s=1}^{n-1}|\eta^{\alpha}_s|^2\geq \delta'_{\varepsilon}$ for $\alpha\geq 2;$ we then consider two cases: 
\\
\\a. If $\gamma_1< 0$, the inequality (\ref{est C2b ineg ILT1}) of Ivochkina, Lin and Trudinger yields
\begin{equation}\label{ILT simple}
\sigma_{m-1,1}\gamma_1^2\leq C_0\sum_{\alpha=2}^n\sigma_{m-1,\alpha}\gamma_{\alpha}^2
\end{equation}
for some constant $C_0=C_0(n,m).$ Since $u_{\alpha\alpha}=\gamma_\alpha-d(\log\mu)(\nabla^Su)$ by (\ref{uab diagonal}), there exist controlled constants $c_1,c_2,c'_1$ and $c'_2$ such that $u_{11}^2\leq c_1\gamma_1^2+c_2$ and $\gamma_\alpha^2\leq c_1'u_{\alpha\alpha}^2+c'_2$ for all $\alpha\geq 2$,  and (\ref{ILT simple}) implies that
$$\sum_{\alpha=1}^n\sigma_{m-1,\alpha}u_{\alpha\alpha}^2\leq C'\sum_{\alpha=2}^n\sigma_{m-1,\alpha}u_{\alpha\alpha}^2+C''\sigma_{m-1}.$$
We deduce the estimates
$$\sum_{\alpha=1}^n\sigma_{m-1,\alpha}|u_{\alpha\alpha}|\leq C_1\sum_{\alpha=2}^n\sigma_{m-1,\alpha}u_{\alpha\alpha}^2+C_2\sigma_{m-1}$$
and 
$$\sum_{\alpha=1}^n\sigma_{m-1,\alpha}|u_{\alpha}u_{\alpha\alpha}W_\alpha|\leq C_3 \sum_{\alpha=1}^n\sigma_{m-1,\alpha}W_\alpha^2+C_4\sum_{\alpha=2}^n\sigma_{m-1,\alpha}u_{\alpha\alpha}^2+C_5\sigma_{m-1}$$
for controlled constants $C_i,$ $i=1,\ldots,5,$ so that (\ref{sum Waa lem interm}) implies (\ref{ineg FijWij lem}) if $K$ is sufficiently large.
\\
\\b. If $\gamma_1\geq 0,$ setting $Q=\sigma_{m-1,1}u_1u_{11}W_1,$ we write
\begin{equation}\label{est C2b eqn est Q 1}
\sum_{\alpha=1}^n\sigma_{m-1,\alpha}u_\alpha u_{\alpha\alpha}W_{\alpha}=Q+\sum_{\alpha=2}^n\sigma_{m-1,\alpha}u_\alpha u_{\alpha\alpha}W_{\alpha}.
\end{equation}
The last term is bounded by
\begin{equation}\label{est C2b eqn est Q 2}
\left|\sum_{\alpha=2}^n\sigma_{m-1,\alpha}u_\alpha u_{\alpha\alpha}W_{\alpha}\right|\leq C\sum_{\alpha=1}^n\sigma_{m-1,\alpha}W_{\alpha}^2+\sum_{\alpha=2}^n\sigma_{m-1,\alpha}u_{\alpha\alpha}^2.
\end{equation}
We finally bound $Q:$ we have $f-\sigma_{m,1}=\sigma_{m-1,1}\gamma_1\geq 0$ and using (\ref{uab diagonal}) we write
\begin{equation}\label{eqn Q 1}
Q=(f-\sigma_{m,1})u_1W_1-d(\log\mu)(\nabla^Su)\sigma_{m-1,1}u_1W_1.
\end{equation}
We first note that
\begin{equation}\label{eqn Q 2}
|u_1W_1|\leq C(\gamma_1+1).
\end{equation}
This is a consequence of (\ref{uab diagonal}) and (\ref{expr Walpha}) for $\alpha=1$. We deduce the estimate
\begin{equation}\label{eqn Q 3}
|d(\log\mu)(\nabla^Su)|\ \sigma_{m-1,1}|u_1W_1|\leq C'\sigma_{m-1,1}(\gamma_1+1).
\end{equation}
Assuming first that $\sigma_{m,1}\geq 0$ we have $\sigma_{m-1,1}\gamma_1=\sigma_m-\sigma_{m,1}\leq\sigma_m$ and we readily obtain from (\ref{eqn Q 1}), (\ref{eqn Q 2}) and (\ref{eqn Q 3}) that
$$|Q|\leq C(1+|\nabla^S W|+\sigma_{m-1}).$$
If now $\sigma_{m,1}<0,$ we have
$$|Q|\leq |u_1||W_1|f-|u_1||W_1|\sigma_{m,1}+|u_1||W_1||d(\log\mu)(\nabla^Su)|\sigma_{m-1,1}.$$
The first term on the right-hand side is bounded by $C_1|\nabla^S W|$ and using (\ref{eqn Q 2}) we have 
\begin{equation}\label{est C2b est Q interm}
|Q|\leq C_1|\nabla^S W|-C_2\sigma_{m,1}(\gamma_1+1)+C_3\sigma_{m-1,1}(\gamma_1+1).
\end{equation}
We will obtain the estimate of $Q$ thanks to the following inequalities:
\begin{enumerate}
\item $-\sigma_{m,1}\gamma_1\leq C'\sum_{\alpha=2}^n\sigma_{m-1,\alpha}u_{\alpha\alpha}^2+C''\sigma_{m-1};$ indeed, writing $\sigma_{m-1,\alpha}\gamma_{\alpha}=\sigma_m-\sigma_{m,\alpha},$ we have
\begin{eqnarray*}
\sum_{\alpha=2}^n\sigma_{m-1,\alpha}\gamma_\alpha^2=\sum_{\alpha=2}^n(\sigma_m-\sigma_{m,\alpha})\gamma_\alpha&=&\sigma_m\sigma_{1,1}-\sum_{\alpha=2}^n\sigma_{m,\alpha}\gamma_\alpha\\
&=&\sigma_m\sigma_{1,1}-(m+1)\sigma_{m+1}+\sigma_{m,1}\gamma_1,
\end{eqnarray*}
and since $\sigma_{m+1}=\sigma_{m,1}\gamma_1+\sigma_{m+1,1}$ we deduce that
$$-\sigma_{m,1}\gamma_1=\frac{1}{m}\left(\sum_{\alpha=2}^n\sigma_{m-1,\alpha}\gamma_\alpha^2+(m+1)\sigma_{m+1,1}-\sigma_m\sigma_{1,1}\right);$$
we then conclude using $\sigma_{m+1,1}\leq C_0\sum_{\alpha=2}^n\sigma_{m-1,\alpha}\gamma_\alpha^2$ (Eq. (\ref{est C2b ineg ILT2})), $\sigma_m\sigma_{m-1,1}\geq 0$ and
\begin{equation}\label{ineq gammaalpha ualpha}
\sum_{\alpha=2}^n\sigma_{m-1,\alpha}\gamma_\alpha^2\leq C_1\sum_{\alpha=2}^n\sigma_{m-1,\alpha}u_{\alpha\alpha}^2+C_2\sigma_{m-1};
\end{equation}
\item $-\sigma_{m,1}=\sigma_{m-1,1}\gamma_1-f\leq \sigma_{m-1,1}\gamma_1;$
\item $\sigma_{m-1,1}\gamma_1\leq C(1+\sigma_{m-1}+\sum_{\alpha=2}^n\sigma_{m-1,\alpha}u_{\alpha\alpha}^2);$ it is a consequence of
\begin{eqnarray}
\sigma_{m-1,1}\gamma_1&=&m\sigma_m-\sum_{\alpha=2}^n\sigma_{m-1,\alpha}\gamma_{\alpha}\nonumber\\
&\leq&m\sigma_m+\frac{1}{4\gamma_0}\sum_{\alpha=2}^n\sigma_{m-1,\alpha}+\gamma_0\sum_{\alpha=2}^n\sigma_{m-1,\alpha}\gamma_\alpha^2\nonumber\\
&\leq&C_{\gamma_0}(1+\sigma_{m-1})+\gamma_0\sum_{\alpha=2}^n\sigma_{m-1,\alpha}\gamma_\alpha^2,\label{est sigma m-11gamma1}
\end{eqnarray}
where $\gamma_0>0$ is arbitrary, together with (\ref{ineq gammaalpha ualpha});
\item $\sigma_{m-1,1}\leq \sigma_{m-1}$ (since $\sigma_{m-1}=\sigma_{m-1,1}+\sigma_{m-2,1}\gamma_1\geq \sigma_{m-1,1}$).
\end{enumerate}
We deduce from (\ref{est C2b est Q interm}) and the four inequalities above that
\begin{equation}\label{est C2b eqn est Q 3}
|Q|\leq C\left(1+|\nabla^S W|+\sigma_{m-1}+\sum_{\alpha=2}^n\sigma_{m-1,\alpha}u_{\alpha\alpha}^2\right).
\end{equation}
(\ref{est C2b eqn est Q 1}), (\ref{est C2b eqn est Q 2}) and (\ref{est C2b eqn est Q 3}) imply that
$$\sum_{\alpha=1}^n\sigma_{m-1,\alpha}|u_{\alpha}u_{\alpha\alpha}W_{\alpha}|\leq C\left(1+|\nabla^SW|+\sigma_{m-1}+\sum_{\alpha=1}^n\sigma_{m-1,\alpha}W_{\alpha}^2+\sum_{\alpha=2}^n\sigma_{m-1,\alpha}u_{\alpha\alpha}^2\right).$$
By the same arguments (especially (\ref{est sigma m-11gamma1})), we also have
$$\sum_{\alpha=1}^n\sigma_{m-1,\alpha}|u_{\alpha\alpha}|\leq C\left(1+\sigma_{m-1}+\sum_{\alpha=2}^n\sigma_{m-1,\alpha}u_{\alpha\alpha}^2\right)$$
and (\ref{sum Waa lem interm}) implies (\ref{ineg FijWij lem}) for $K$ sufficiently large. This finishes the proof of Lemma \ref{lem g ineg FijWij}.
\begin{lem}\label{lem ineq Wtilde}
For a constant $b$ sufficiently large, the function
\begin{equation}\label{def Wtilde estC2b}
\widetilde{W}=\exp\left(-C_1g(x_0,\nabla^S u(x_0))\right)-\exp(-C_1W)-b\ d_S(x_0,.)^2
\end{equation}
satisfies
\begin{equation}\label{ineq Wtildeij estC2b}
\sum_{i,j}\frac{\partial\mathcal{G}_m}{\partial q_{ij}}[u]\ \widetilde{W}'_{ij}\leq C_2(1+|\nabla^S\widetilde{W}|)
\end{equation}
where
$$\widetilde{W}'_{ij}:=\widetilde{W}_{ij}+(\delta_{ij}-u_iu_j)\ d\widetilde{W}(\nabla^S(\log\mu))$$
and $d_S(x_0,.)$ stands for the function distance to $x_0$ with respect to the metric $g_S.$
\end{lem}
\begin{proof}
Let us consider the function $\phi=r^2$ where $r=d_S(x_0,.)$, and fix $\delta_0,r_0>0$ small such that $\nabla^Sd\phi\geq \delta_0 g_S$ on $\overline{B}_{r_0}(x_0).$ We have
$$(\phi'_{ij})_{ij}:=(\phi_{ij})_{ij}+d\phi(\nabla^S(\log\mu))(\delta_{ij}-u_iu_j)_{ij}\geq(\delta_0-cr)g_S$$
where $c$ is a constant such that $|d\phi(\nabla^S(\log\mu))|=2r|dr(\nabla^S(\log\mu))|\leq cr,$ and choosing $r_0$ smaller we may suppose that
$$(\phi'_{ij})_{ij}\geq\frac{\delta_0}{2} g_S.$$
From that inequality and (\ref{ineg FijWij lem}) we deduce that
\begin{eqnarray*}
\sum_{i,j}\frac{\partial\mathcal{G}_m}{\partial q_{ij}}[u]\ \widetilde{W}'_{ij}&\leq& C_1^2\exp(-C_1W)\left(1+|\nabla^SW|+\mathcal{G}_{m-1}[u]\right)-b\sum_{i,j}\frac{\partial\mathcal{G}_m}{\partial q_{ij}}[u]\phi'_{ij}\\
&\leq& C_1^2\exp(-C_1W)\left(1+|\nabla^SW|+\mathcal{G}_{m-1}[u]\right)-\frac{b\delta_0}{2}\sum_{i}\frac{\partial\mathcal{G}_m}{\partial q_{ii}}[u]\\
&\leq&C(1+|\nabla^SW|)
\end{eqnarray*}
if $b$ is sufficiently large. Finally, since $dW=C_1^{-1}\exp(C_1W)(d\widetilde{W}+2brdr)$ we have $|\nabla^SW|\leq  C'(1+|\nabla^S\widetilde{W}|)$ and (\ref{ineq Wtildeij estC2b}) holds for some controlled constant $C_2.$
\end{proof}
We will use the following elementary comparison principle:
\begin{lem}\label{est C2b ppe comp w'}
Let $u,v,w$ be $C^2$ on $\overline{\Omega},$ with $u$ admissible, such that
\begin{equation}\label{est C2b ppe max ineq w}
\sum_{i,j}\frac{\partial\mathcal{G}_m}{\partial q_{ij}}[u]\ w'_{ij}\leq h(.,\nabla^S w)
\end{equation}
and 
\begin{equation}\label{est C2b ppe max ineq v}
\sum_{i,j}\frac{\partial\mathcal{G}_m}{\partial q_{ij}}[u]\ v'_{ij}> h(.,\nabla^S v)
\end{equation}
in $\Omega,$ where
$$w'_{ij}:=w_{ij}+(\delta_{ij}-u_iu_j)dw(\nabla^S(\log\mu))\hspace{.3cm}\mbox{and}\hspace{.3cm}v'_{ij}:=v_{ij}+(\delta_{ij}-u_iu_j)dv(\nabla^S(\log\mu))$$
and $h:\Omega\times B(0,1)\rightarrow\R$ is a real function. 
If $v\leq w$ on $\partial\Omega$ then $v\leq w$ in $\overline{\Omega}.$
\end{lem}
\begin{proof}
Assuming by contradiction that $v>w$ somewhere in $\Omega,$ $v-w$ would reach its maximum at some point $x_0\in\Omega,$ and we would have $dv_{x_0}=dw_{x_0}$ and $(v_{ij}(x_0))_{ij}\leq (w_{ij}(x_0))_{ij};$ that would imply $(v'_{ij}(x_0))_{ij}\leq (w'_{ij}(x_0))_{ij},$ and yield a contradiction with (\ref{est C2b ppe max ineq w}) and (\ref{est C2b ppe max ineq v}).
\end{proof}
\begin{lem}
Let us suppose that $\psi\in C^2(\partial\Omega\cap\overline{B}_{r}(x_0))$ and $a_0\in\R$ are given, and let us denote by $d$ the distance to $\partial\Omega$ with respect to the metric $g_S.$ Then, for large parameters $b_0$ and $c_0,$ the function
\begin{equation}\label{def v estC2b}
v=-a_0d_S(x_0,.)^2-h(d)+\psi(x')
\end{equation}
with $h(d)=c_0(1-\exp(-b_0d))$ is such that $(v'_{ij})_{ij}$ is positive and satisfies
\begin{equation}\label{est C2b ineq vpij}
\sum_{i,j}\frac{\partial\mathcal{G}_m}{\partial q_{ij}}[u]\ v'_{ij}> C_2(1+|\nabla^S v|)
\end{equation}
on $\Omega\cap B_r(x_0).$
\end{lem}
Here and below we suppose that $x=(x',x_n)$ in a basis $e_1^0,\ldots,e_n^0$ such that $e_1^0,\ldots,e_{n-1}^0$ are tangent and $e_n^0$ is normal to $\partial\Omega$ at $x_0,$ and, if $x_n=\rho(x')$ is a local description of $\partial\Omega$, we denote by $\psi(x')$ the function $\psi(x',\rho(x')).$
\begin{proof}
Let us first recall the following inequalities: for $\mathcal{F}_m(p,q')=F_m(A^{-1}(p)q'),$ the concavity of $\mathcal{F}_m^{\frac{1}{m}}$ with respect to $q'$ in the convex cone $\Gamma_m(p)$ implies that
$$\sum_{i,j}\frac{\partial\mathcal{F}_m}{\partial q'_{ij}}(p,q')\ r'_{ij}\geq m\ \mathcal{F}_m^{1-\frac{1}{m}}(p,q')\ \mathcal{F}_m^{\frac{1}{m}}(p,r')$$  
for all $q',r'\in\Gamma_m(p)$, with 
\begin{equation}\label{est C2b cal Fm geq Fm}
\mathcal{F}_m(p,r')\geq F_m(r')
\end{equation}
if $r'$ is positive; see \cite[Lemma 4.7]{Bay1} for that last inequality. With $q'=q+d(\log\mu)_x(p)\ A(p)$ and $\mathcal{G}_m(x,p,q)=\mathcal{F}_m(p,q')$ it implies the following: if $u$ is admissible and $(v'_{ij})_{ij}$ is positive then
\begin{equation}\label{sum F_muvijFmv}
\sum_{i,j}\frac{\partial \mathcal{G}_m}{\partial q_{ij}}[u]\ v'_{ij}\ \geq\ m\ \mathcal{G}_m^{1-\frac{1}{m}}[u]\ \mathcal{F}_m^{\frac{1}{m}}(\nabla^Su,(v'_{ij})_{ij})\ \geq\ cF_m^{\frac{1}{m}}[(v'_{ij})_{ij}]
\end{equation}
for some positive constant $c$, where $F_m[(v'_{ij})_{ij}]$ stands for the $m^{th}$ symmetric function of the eigenvalues of $(v'_{ij})_{ij}$. Let us set $\Omega_a:=\{x\in\Omega|\ d_S(x,\partial\Omega)<a\}.$ If $a$ is sufficiently small, for all $x\in \Omega_a$ there exists a unique $y\in\partial\Omega$ such that $d_S(x,y)=d_S(x,\partial\Omega).$ Let us fix $x\in\Omega_a$ and consider $e_1,\ldots,e_{n-1}$ an orthonormal basis of principal directions of $\partial\Omega$ at the point $y$ corresponding to $x,$ completed with the inward unit normal $e_n$ at $y.$ Then the hessian of $d$ in the orthonormal basis at $x$ obtained from $e_1,\ldots,e_n$ by parallel transport along the geodesic between $y$ and $x$ is the diagonal matrix
\begin{equation}\label{curv tub}
\mbox{diag}\left(-\frac{\kappa'_i\cos d+\sin d}{\cos d-\kappa'_i\sin d},\ 1\leq i\leq n-1,\ 0\right)
\end{equation}
where $\kappa',\ldots,\kappa'_{n-1}$ are the principal curvatures of $\partial\Omega$ with respect to the inward normal at $y.$ See for instance \cite[Chapter 6.3]{Gr}, and especially Theorem 6.14, from which one can see that the principal curvatures of a tubular hypersurface at distance $d$ in $\mathbb{S}^n$ are obtained by integration of the equation $z'=z^2+1$ between $0$ and $d$, which easily yields (\ref{curv tub}) (noticing that the hessian of the distance function is the opposite of the second fundamental form of the boundary w.r.t. the inner normal). Denoting here $r=d_S(x_0,.)$, we have 
\begin{equation}\label{nabla v lem1 v est C2b}
\nabla^S v=-a_0\nabla^S(r^2)-h'\nabla^S d+\nabla^S\psi
\end{equation}
and 
\begin{equation}\label{maj nabla v lem1 v est C2b}
|\nabla^S v|\leq a_0|\nabla^S(r^2)|+h'+|\nabla^S\psi|\leq 3h'
\end{equation} 
if $h'$ is large. Moreover,
$$\nabla^S dv=h'\left(-\frac{a_0}{h'}\nabla^S d (r^2)+D+\frac{1}{h'}\nabla^S d(\psi(x'))\right)$$
where $D$ is the diagonal matrix
$$D=\mbox{diag}\left(\frac{\kappa'_i\cos d+\sin d}{\cos d-\kappa'_i\sin d},\ 1\leq i\leq n-1,\ b_0\right).$$
Let us suppose that $b_0\geq 1$ and fix
$$\delta_0=\frac{1}{2}\inf_{\partial\Omega}\left(\min(\kappa'_1,\ldots,\kappa'_{n-1},1)\right).$$
Since $D$ tends uniformly to $\mbox{diag}(\kappa'_1,\ldots,\kappa'_{n-1},b_0)$ as $d$ tends to 0, there exists $a>0$ such that if $d\leq a$ then $D\geq\delta_0\ \mbox{diag}(1,\ldots,1,b_0).$ We deduce that $(v_{ij})_{ij}\geq h'\frac{\delta_0}{2}\ \mbox{diag}(1,\ldots,1,b_0)$ if $h'$ is large. Let us note that $d\mu(\nabla^Sd)=\mu_n<0$ on $\partial\Omega$ by Lemma \ref{app mun negatif} in the appendix, and therefore that $d\mu(\nabla^Sd)\leq 0$ on $\Omega_a$ if $a$ is small. Using (\ref{nabla v lem1 v est C2b}) we deduce that
$$d(\log\mu)(\nabla^Sv)\ A(\nabla^Su)\geq d(\log\mu)(-a_0\nabla^S(r^2)+\nabla^S\psi)\ A(\nabla^Su)\geq-cI.$$
It implies that
\begin{eqnarray*}
(v'_{ij})_{ij}=\nabla^Sdv+d(\log\mu)(\nabla^Sv)\ A(\nabla^Su)&\geq& \nabla^Sdv-cI\\
&\geq& h'\left(\frac{\delta_0}{2}\ \mbox{diag}(1,\ldots,1,b_0)-\frac{c}{h'}I\right)\\
&\geq& h'\frac{\delta_0}{4}\ \mbox{diag}(1,\ldots,1,b_0)
\end{eqnarray*}
if $h'$ is sufficiently large. In particular $(v'_{ij})_{ij}$ is positive and 
$$F_m[(v'_{ij})_{ij}]\geq h'^m\frac{\delta_0^m}{4^m}\sigma_m(1,\ldots,1,b_0)\geq h'^m\frac{\delta_0^m}{4^m}b_0\sigma_{m-1}(1,\ldots,1)\geq C h'^m$$
where $C$ is a constant as large as desired. Finally, since $|\nabla^S v|\leq 3h'$ by (\ref{maj nabla v lem1 v est C2b}), we have $h'=\frac{1}{2}h'+\frac{1}{2}h'\geq \frac{1}{2}+\frac{1}{6}|\nabla^S v|\geq \frac{1}{6}(1+|\nabla^S v|)$ and
$$F_m[(v'_{ij})_{ij}]\geq C\left(1+|\nabla^S v|\right)^m,$$
which, in view of (\ref{sum F_muvijFmv}), gives the result (\ref{est C2b ineq vpij}).
\end{proof}
\subsubsection{Estimates of the mixed second derivatives}
We suppose that $e_1^0,\ldots,e_n^0$ is an orthonormal basis of the euclidean space $\R^n$ such that $e_1^0,\ldots,e_{n-1}^0$ are tangent to $\partial\Omega$ at $x_0$ and $e_n^0$ is inward-directed and normal to $\partial\Omega$ at $x_0$. Let $\rho:\R^{n-1}\rightarrow\R$ be such that $x_n=\rho(x')$ describes $\partial\Omega$ near $x_0.$ We have $x_0=(x_0',\rho(x_0'))$ and $d\rho_{x_0'}=0.$ We define the basis $e_i=1/\lambda\ e_i^0,$ $i=1,\ldots,n,$ which is orthonormal with respect to $g_S$ (and depends on the position $x$), and set
$$\xi(x)=e_t+\rho_t(x')e_n,\hspace{.5cm}g(x,p)=g_S(p,\xi(x))=p_t+\rho_t(x')p_n$$
and define $W$ by (\ref{def W der sec bord}) and $\widetilde{W}$ by (\ref{def Wtilde estC2b}). Since $u=\varphi$ on $\partial\Omega$ we have $u_s+\rho_su_n=\varphi_s+\rho_s\varphi_n$ on $\partial\Omega$ and
$$\sum_{s=1}^{n-1}(u_s-u_s(x_0))^2\leq 2\sum_{s=1}^{n-1}(\varphi_s-\varphi_s(x_0))^2+4(|\varphi_n|^2+1)|\nabla^S\rho|^2$$
where $|\nabla^S\rho|^2=\sum_{s=1}^{n-1}(\rho_s)^2.$ So, setting
\begin{eqnarray*}
\psi(x')&=&\exp(-C_1\varphi_\xi(x_0))\\
&&-\exp(-C_1\varphi_\xi(x'))\exp\left(C_1K\sum_{s=1}^{n-1}(\varphi_s(x')-\varphi_s(x_0))^2+2C_1K(|\varphi_n(x')|^2+1)|\nabla^S\rho(x')|^2\right)
\end{eqnarray*}
where, for a function $h$ defined on a neighborhood of $x_0$, we use $h(x')$ to denote $h(x',\rho(x'))$, for constants $a_0,b_0$ and $c_0$ sufficiently large the function $v$ defined by (\ref{def v estC2b}) is such that:
\begin{itemize}
\item $v(x_0)=\widetilde{W}(x_0);$
\item $v\leq \widetilde{W}$ on $\partial(\Omega\cap B_r(x_0));$
\item $\displaystyle{\sum_{i,j}\frac{\partial\mathcal{G}_m}{\partial q_{ij}}[u]\ v'_{ij}> C_2(1+|\nabla^Sv|)}$ in $\Omega\cap B_r(x_0).$
\end{itemize}
We deduce from Lemma \ref{est C2b ppe comp w'} that $v\leq \widetilde{W}$ in $\overline{\Omega}\cap \overline{B}_r(x_0),$ which implies the estimate 
\begin{equation}\label{est unn vn Wn}
v_n(x_0)\leq\widetilde{W}_n(x_0).
\end{equation} 
Since $\partial_{e_n}(g(x,\nabla^Su(x)))=u_{tn}+\mathcal{C}$, where $u_{tn}=\nabla^Sdu(e_t,e_n)$ and $\mathcal{C}$ is a controlled term, we obtain a lower bound for $u_{tn}$ at $x_0.$ An upper bound of $u_{tn}$ is obtained similarly by taking $\xi(x)=-e_t+\rho_t(x')e_n.$
$ $\\
\subsubsection{Estimates of the double normal second derivatives} We still write the equation of prescribed curvature in the form (\ref{def Fmxpq estC2b})-(\ref{estC2b eqn Gm=f}). Let us denote by $\gamma$ the inner unit vector field normal to the boundary and by $A_m$ the coefficient of $u_{\gamma\gamma}$ in $\mathcal{G}_m(x,\nabla^Su,\nabla^Sdu).$ For $x\in\partial\Omega,$ if $e_1,\ldots,e_n$ is an orthonormal basis of $T_xB^n$ (with respect to $g_S$) such that $e_1,\ldots,e_{n-1}$ are tangent to $\partial\Omega$ and $e_n=\gamma,$ we have
$$A_m= \frac{1-|p_u'|^2}{1-|p_u|^2}\ \mathcal{F}_{m-1}(p_u',q_u')$$
where $p_u=(u_i)_{1\leq i\leq n}$, $p'_u=(u_i)_{1\leq i\leq n-1}$ and 
\begin{eqnarray*}
q_u'&=&\left(u_{ij}+d(\log\mu)(\nabla^S u)\left(\delta_{ij}-u_iu_j\right)\right)_{1\leq i,j\leq n-1}\\
&=&\nabla^Sdu_{|T\partial\Omega}+d(\log\mu)(\nabla^S u)\ g_{\partial\Omega};
\end{eqnarray*}
here $\displaystyle{g_{\partial\Omega}=(\delta_{ij}-\varphi_i\varphi_j)_{1\leq i,j\leq n-1}}$ only depends on the boundary data. Denoting by $\partial^S$ the covariant derivative induced on $\partial\Omega$ by the metric $g_S$ and by $II^S_{\partial\Omega}$ the second fundamental form of $\partial\Omega$ (still computed with respect to $g_S$), since $u=\varphi$ on $\partial\Omega$ we have
$$\nabla^Su=\partial^S\varphi+u_\gamma \gamma\hspace{.5cm}\mbox{and}\hspace{.5cm}\nabla^S du_{|T\partial\Omega}=\partial^Sd\varphi-u_\gamma II^S_{\partial\Omega},$$
so that
$$A_m=\frac{1-|\partial^S\varphi|^2}{1-|\nabla^S u|^2}\ \mathcal{F}_{m-1}\left(\partial^S\varphi,\partial^Sd\varphi+d(\log\mu)(\partial^S\varphi)\ g_{\partial\Omega}-u_\gamma\left(II^S_{\partial\Omega}-\frac{\mu_\gamma}{\mu}g_{\partial\Omega}\right)\right).$$
Let us first note that the obtention of a lower bound of $u_{\gamma\gamma}$ is straightforward: since $u$ is admissible, we have $\mathcal{G}_1[u]>0$, and the coefficient in $\mathcal{G}_1[u]$ of $u_{\gamma\gamma}$ is
$$A_1=\frac{1-|\partial^S\varphi|^2}{1-|\nabla^S u|^2},$$
which is positive and bounded from below, meanwhile the other second derivatives of $u$ which appear in $\mathcal{G}_1[u]$ are bounded (by the estimates of the double tangent and mixed second derivatives obtained above). We focus on the obtention of an upper bound of $u_{\gamma\gamma}.$ We set
\begin{equation*}
g(x,p)= \mathcal{F}_{m-1}\left(\partial^S\varphi,\partial^Sd\varphi+d(\log\mu)(\partial^S\varphi)\ g_{\partial\Omega}-g_S(p,\gamma)\left(II^S_{\partial\Omega}-\frac{\mu_\gamma}{\mu}g_{\partial\Omega}\right)\right)
\end{equation*}
for all $x\in\partial\Omega$ and all $p\in T_xB^n$ such that $g_S(p,p)<1.$ At that point, since we need a concave map $p\mapsto g(x,p)$ we restrict to the case $m=2$: the map $p\mapsto g(x,p)$ is then linear and in particular concave. We use the following observation of Trudinger \cite{ILT,Tr95}:  it is sufficient to bound $u_{\gamma\gamma}$ at a point $y\in\partial\Omega$ which minimizes $x\mapsto g(x,\nabla^Su(x))$ on $\partial\Omega$ in order to obtain a global bound of $u_{\gamma\gamma}$ on $\partial\Omega$. See also \cite{Bay1} for the Dirichlet problem in $\R^{n,1}.$ We suppose that on a neighborhood of $y$ the boundary of $\Omega$ is described as $\partial\Omega=\{(x',x_n),\ x_n=\rho(x')\}$ where $\rho:\R^{n-1}\rightarrow\R$ is a function defined locally so that, for $y=(y',\rho(y'))$, $d\rho_{y'}=0$. We extend $\gamma$ on a neighborhood $\overline{\Omega}\cap B(y,r)$ of $y$ in a unit vector field such that $\nabla^S_{\gamma}\gamma(y)=0$ and also extend $g$ by the formula 
\begin{eqnarray}
g(x,p)&=& \mathcal{F}_{1}\left(\partial^S\varphi,\partial^Sd\varphi+d(\log\mu)(\partial^S\varphi)\ g_{\partial\Omega}\right)(x',\rho(x'))\label{def g estC2b nn ext}\\
&&-g_S(p,\gamma(x))\ \mathcal{F}_1\left(\partial^S\varphi,II^S_{\partial\Omega}-(\log\mu)_\gamma\ g_{\partial\Omega}\right)(x',\rho(x'))\nonumber
\end{eqnarray}
for all $x\in \overline{\Omega}\cap B(y,r)$ and all $p\in T_xB^n$ such that $g_S(p,p)<1.$ As above we consider the test-function
$$W(x)=g(x,\nabla^S u(x))-\frac{K}{2}\sum_{s=1}^{n-1}(u_s(x)-u_s(y))^2$$
and
$$\widetilde{W}(x)=\exp\left(-C_1g(y,\nabla^S u(y))\right)-\exp\left(-C_1W(x)\right)-bd_S(x,y)^2.$$
Since $y$ minimizes $g(x,\nabla^Su(x))$ on $\partial\Omega,$ we have the bound on $\partial\Omega$
$$\widetilde{W}\geq \exp\left(-C_1g(y,\nabla^S u(y))\right)\left(1-\exp\left(C_1\frac{K}{2}\sum_{s=1}^{n-1}(u_s(x)-u_s(y))^2\right)\right)-bd_S(x,y)^2.$$
By Lemma \ref{lem ineq Wtilde}, if $b$ is large $\widetilde{W}$ is such that
$$\sum_{i,j}\frac{\partial\mathcal{G}_2}{\partial q_{ij}}[u]\ \widetilde{W'}_{ij}\leq C_2(1+|\nabla^S\widetilde{W}|).$$
For the barrier function we take $v=-a_0d_S(.,y)^2-h(d)+\psi(x')$
with $h(d)=c_0(1-e^{-b_0d})$ and
\begin{eqnarray*}
\psi(x')&=&\exp\left(-C_1g(y,\nabla^S u(y))\right)\times\\
&&\left\{1-\exp \left(C_1K\sum_{s=1}^{n-1}(\varphi_s(x')-\varphi_s(y))^2+2C_1K\left(|\varphi_n(x')|^2+1\right)|\nabla^S\rho(x')|^2\right)\right\}.
\end{eqnarray*}
If $a_0,b_0$ and $c_0$ are sufficiently large we have:
\begin{itemize}
\item $v(y)=\widetilde{W}(y);$
\item $v\leq \widetilde{W}$ on $\partial(\Omega\cap B_r(y));$
\item $\displaystyle{\sum_{i,j}\frac{\partial\mathcal{G}_2}{\partial q_{ij}}[u]\ v'_{ij}> C_2(1+|\nabla^Sv|)}$ in $\Omega\cap B_r(y).$
\end{itemize}
We deduce that $v\leq \widetilde{W}$ in $\Omega\cap B_r(y),$ which implies the estimate 
\begin{equation}\label{est unn vn Wn}
v_n(y)\leq\widetilde{W}_n(y).
\end{equation} 
By (\ref{def g estC2b nn ext}) we have at $y$
$$\partial_{e_n}(g(x,\nabla^Su(x)))=-u_{\gamma\gamma}\ \mathcal{F}_1\left(\partial^S\varphi, II^S_{\partial\Omega}-(\log\mu)_\gamma\ g_{\partial\Omega}\right)$$
and (\ref{est unn vn Wn}) implies that
$$u_{\gamma\gamma}\ \mathcal{F}_{1}\left(\partial^S\varphi,II^S_{\partial\Omega}-(\log\mu)_\gamma\ g_{\partial\Omega}\right)\leq C$$
for some controlled constant $C.$ Since $\Omega$ is strictly convex with respect to $g_S$ and $0$ belongs to $\Omega,$ we have $\mu_\gamma\leq 0$ (Lemma \ref{app mun negatif}) and the quadratic form $II^S_{\partial\Omega}-\frac{\mu_\gamma}{\mu}g_{\partial\Omega}$ is positive definite, which implies by (\ref{est C2b cal Fm geq Fm})
$$\mathcal{F}_{1}\left(\partial^S\varphi,II^S_{\partial\Omega}-\frac{\mu_\gamma}{\mu}g_{\partial\Omega}\right)\geq F_1\left(II^S_{\partial\Omega}\right)\geq\inf_{\partial\Omega}F_1\left(II^S_{\partial\Omega}\right)>0.$$
An upper bound of $u_{\gamma\gamma}$ follows.

\appendix
\section{Computations on second fundamental forms}\label{appendix computations}
\subsection{Second fundamental form of a graph}
\begin{lem}
For a spacelike hypersurface $M=\hbox{graph}(u),$ in the chart $x\in B^n\mapsto (x,u(x))\in M$ the upward unit normal vector is given by
\begin{equation}\label{expression N}
N=\frac{1}{\mu}\frac{1}{\sqrt{1-|\nabla^S u|^2}}(\nabla^S u, 1)
\end{equation}
and the second fundamental form by
\begin{equation}\label{lem expr h u}
h=\frac{\mu}{\sqrt{1-|\nabla^S u|^2}}\left(\nabla^Sdu+d(\log\mu)(\nabla^S u)\ (g_S-du\otimes du)\right)
\end{equation}
where $\nabla^Su$ and $\nabla^Sdu$ stand for the gradient and the hessian of $u$ with respect to the metric $g_S.$
\end{lem}
\begin{proof}
The vector
\begin{equation}\label{expression N'}
N'=\frac{1}{\sqrt{1-|\nabla^S u|^2}}(\nabla^S u, 1)
\end{equation}
is the unit vector normal to $M$ with respect to the product metric $g'=g_S-dt^2$ on $B^n\times\R$. Since $g=\mu^2 g'$ we have $N=1/\mu\ N'$ and we obtain (\ref{expression N}). Let us denote by $\widetilde{\nabla}'$ the Levi-Civita connection of $g'$ and set $S'(X)=\widetilde{\nabla}'_XN'$ ($N'$ and $S'$ are respectively the unit normal and the shape operator of $M$ with respect to the metric $g'$). Since $g=\mu^2 g',$ by the Koszul formula we have
$$\widetilde{\nabla}_XY-\widetilde{\nabla}'_XY=X.\log\mu\ Y+Y.\log\mu\ X-g'(X,Y)\widetilde{\nabla}'(\log\mu)$$
where $\widetilde{\nabla}'(\log\mu)$ is the gradient of $\log\mu$ with respect to $g'.$ It implies that
$$\widetilde{\nabla}_XN=\frac{1}{\mu}\widetilde{\nabla}'_XN'+\frac{1}{\mu^2}N'.\mu\ X$$
for all $X\in TM,$ which reads $S=\frac{1}{\mu}\left(S'+d(\log\mu)(N')\ id\right);$ setting $h'(X,Y)=g'(S'(X),Y)$ we deduce that
\begin{equation}\label{relation h h'}
h(X,Y)=\mu h'(X,Y)+d\mu(N')g'.
\end{equation}
We first compute $h'.$ We consider the basis of $TM$ induced by the chart $x\mapsto (x,u(x))$, 
\begin{equation}\label{expression partial i u}
\partial_i=e_i+u_i e_{n+1},\hspace{.5cm}i=1,\ldots,n.
\end{equation}
Here $e_1,\ldots,e_n$ is a basis of $TB^n$ and $e_{n+1}$ spans the $\R$-factor of $\widetilde{AdS}^{n,1}=B^n\times\R.$ By (\ref{expression N'}) we have
$$S'(\partial_i)=\widetilde{\nabla}'_{\partial_i}N'=\frac{1}{\sqrt{1-|\nabla^S u|^2}}\left(\widetilde{\nabla}'_{\partial_i}(\nabla^S u,1)\right)^{T'},$$
where the exponent $T'$ means that we take the tangential component of the vector with respect to $g'.$ Since $\widetilde{\nabla}'$ is the Levi-Civita connection with respect to the product metric $g'=g_S-dt^2$ we deduce that 
\begin{equation}\label{expression h'}
h'_{ij}=g'(S'(\partial_i),\partial_j)=\frac{1}{\sqrt{1-|\nabla^S u|^2}}\nabla^S_{ij}u.
\end{equation}
In that formula $\nabla^S_{ij}u$ stands for $\nabla^Sdu(e_i,e_j),$ where $\nabla^Sdu$ is the hessian of $u$ with respect to $g_S.$ We now observe that the second term in (\ref{relation h h'}) is given by
\begin{equation}\label{expression dmu N'}
d\mu(N')=\frac{1}{\sqrt{1-|\nabla^S u|^2}}d\mu(\nabla^S u),
\end{equation}
by (\ref{expression N'}) and since $\mu$ is independent of the second factor of $B^n\times\R$. So, using (\ref{expression h'}) and (\ref{expression dmu N'}), (\ref{relation h h'}) reads
$$h_{ij}=\mu h'_{ij}+d\mu(N')g'_{ij}=\frac{\mu}{\sqrt{1-|\nabla^S u|^2}}\left(\nabla^S_{ij}u+d(\log\mu)(\nabla^S u)\ g'_{ij}\right).$$
Finally, using (\ref{expression partial i u}) we get $g'_{ij}=(g_S-dt^2)(\partial_i,\partial_j)=g_{S}(e_i,e_j)-u_iu_j$ and obtain (\ref{lem expr h u}).
\end{proof}
\subsection{Comparison of second fundamental forms}\label{app rel ff}
Let us suppose that $M$ is a spacelike hypersurface of $\widetilde{AdS}^{n,1}$ such that $\partial M$ is the graph of $\varphi:\partial \Omega\subset B^n\rightarrow\R$ and denote by \begin{eqnarray*}
F:\partial\Omega&\rightarrow&\partial M\subset B^n\times\R\\
x&\mapsto& (x,\varphi(x))
\end{eqnarray*}
the natural parametrization of $\partial M$ and by $n'\in\R^n$ the vector field normal to $\partial\Omega$, inward-directed and such that $g_H(n',n')=1.$ 
\begin{lem}\label{app rel ff}
The formula
$$\langle II_{\partial M}(F_*X,F_*X),n'\rangle=II_{\partial\Omega}^{H}(X,X)+\mu\mu_{n'}d\varphi(X)^2$$
holds for all $X\in T\partial\Omega$, where $II_{\partial M}$ is the second fundamental form of $\partial M$ in $\widetilde{AdS}^{n,1}$.
\end{lem}
\begin{proof}
Let us consider $X,Y\in\Gamma(\partial\Omega)$ and denote $\overline{X}=F_*X=X+d\varphi(X)e_{n+1}$ and  $\overline{Y}=F_*Y=Y+d\varphi(Y)e_{n+1}\in\Gamma(\partial M)$. We have
$$\langle II_{\partial M}(\overline{X},\overline{Y}),n'\rangle=\langle \widetilde{\nabla}_{\overline{X}}\overline{Y},n'\rangle=\langle\widetilde{\nabla}_{\overline{X}}Y,n'\rangle+d\varphi(Y)\langle \widetilde{\nabla}_{\overline{X}}e_{n+1},n'\rangle.$$
We compute the first term on the right-hand side using
$$\widetilde{\nabla}_{\overline{X}}Y=\widetilde{\nabla}_{X}Y+d\varphi(X)\widetilde{\nabla}_{e_{n+1}}Y=\nabla^H_{X}Y+\frac{1}{\mu}d\mu(Y)d\varphi(X)e_{n+1}$$
and the second term using
$${\widetilde{\nabla}}_{\overline{X}}e_{n+1}=\widetilde{\nabla}_{X}e_{n+1}+d\varphi(X)\widetilde{\nabla}_{e_{n+1}}e_{n+1}=\frac{1}{\mu}d\mu(X)e_{n+1}+d\varphi(X)\mu\nabla^H\mu$$
(these formulas are consequences of (\ref{notation C1 dev cov en+1}) and (\ref{notation C1 dev cov Y})) and obtain
$$\langle II_{\partial M}(\overline{X},\overline{Y}),n'\rangle=\langle\nabla^H_{X}Y,n'\rangle+\mu d\varphi(X)d\varphi(Y)\langle\nabla^H\mu,n'\rangle,$$
which gives the formula.
\end{proof}
\begin{lem}\label{app mun negatif}
Let us suppose that $\overline{\Omega}\subset B^n$ is convex with respect to the metric $g_H$ and contains the center $0$ of $B^n.$ If $n'$ is the unit inner vector field normal to $\partial\Omega$ then $\mu_{n'}:=d\mu(n')\leq 0$ on $\partial\Omega.$ Moreover, if 0 belongs to $\Omega$ then $\mu_{n'}<0.$ The same result holds if we consider the metric $g_S$, or the euclidian metric, instead of $g_H.$
\end{lem}
\begin{proof} 
Let us first observe that if $\gamma$ is a geodesic such that $\gamma(0)=p\in\overline{\Omega}$ and $\gamma(1)=q\in\partial\Omega,$ we have $\langle\gamma'(1),n'\rangle\leq 0,$ and that the inequality is strict if $p$ belongs to $\Omega$. Indeed, applying an isometry we may suppose that $q=0$, $T_0\partial\Omega=\{x_n=0\}$ and $\overline{\Omega}\subset\{x_n\leq 0\},$ so that $p=(p_1,\ldots,p_n)$ with $p_n\leq 0.$ Moreover, the geodesic $\gamma$ is of the form $\gamma(t)=\alpha(t)p,$ with $\alpha(0)=1,$ $\alpha(1)=0$ and $\alpha'\leq 0$. Since $\gamma'(1)=\alpha'(1)p$ and $n'=-e_n$ at $q=0$, we obtain $\langle\gamma'(1),n'\rangle=-\alpha'(1)p_n\leq 0$. If $p$ belongs to $\Omega$ then $p_n<0$ and $\alpha'(1)<0$ ($\gamma'$ does not vanish since $p\neq q$ in that case and $\gamma$ is a non trivial geodesic), and the inequality is strict.

Let us assume now that $p=0\in\overline{\Omega}$, $q\in\partial\Omega$ with $q\neq p,$ and consider the geodesic between $p$ and $q,$ $\gamma(t)=\beta(t)q$ with $\beta(0)=0$ and $\beta(1)=1$. By the observation above we have $\langle\gamma'(1),n'\rangle=\beta'(1)\langle q,n'\rangle\leq 0,$ so that $\langle q,n'\rangle\leq 0$, with $\langle q,n'\rangle<0$ if $p\in\Omega.$ A direct computation shows that $\nabla^H\mu(x)=x$ for all $x\in B^n$, and we deduce that
$$\mu_{n'}(q)=\langle \nabla^H\mu\ (q),n'\rangle=\langle q,n'\rangle\leq 0,$$
and that $\mu_{n'}(q)<0$ if $p$ belongs to $\Omega.$ The proofs for the metric $g_S$ or for the euclidian metric are analogous.
 \end{proof}
Lemmas \ref{app rel ff} and \ref{app mun negatif} readily imply the following:
\begin{lem}\label{app lem ff} 
Keeping the notation of Lemma \ref{app rel ff} and assuming moreover that $\overline{\Omega}$ is convex with respect to the hyperbolic metric $g_H$ and contains the center $0$ of $B^n$ we have
$$II_{\partial\Omega}^{H}(X,X)\geq \langle II_{\partial M}(F_*X,F_*X),n'\rangle$$
for all $X\in T\partial\Omega.$
\end{lem}
We also deduce the following result:
\begin{lem}\label{app lem tt geod H2}
If $\Omega$ is strictly convex with respect to $g_S$ and contains the center $0$ of $B^n$ then every totally geodesic boundary data $M_{\varphi}$ satisfies the convexity assumption (H2). 
\end{lem}
\begin{proof}
If $M_{\varphi}$ is totally geodesic, the second fundamental form of $\partial M_{\varphi}$ in $\widetilde{AdS}^{n,1}$ reduces to the second fundamental form of $\partial M_{\varphi}$ in $M_{\varphi}$, i.e. is given by
\begin{equation}\label{pf lem tt geod H2 II}
II_{\partial M_{\varphi}}=II_{\partial M_{\varphi},M_{\varphi}}N_1
\end{equation}
where $N_1$ is the unit vector which is normal to $\partial M_{\varphi}$, tangent to $M_{\varphi}$ and inward-directed. By Lemma \ref{app rel ff} and using the relation (\ref{II partial Omega gS gH}) between $II^H_{\partial\Omega}$ and $II^S_{\partial\Omega}$ we obtain
$$II_{\partial M_{\varphi},M_{\varphi}}(F_*X,F_*X)\ \langle N_1,n'\rangle=\mu\ II^S_{\partial\Omega}(X,X)-\mu_n\left(g_S(X,X)-d\varphi(X)^2\right)$$
for all $X\in T\Omega,$ where we also use that $n=\mu n'$ is the unit normal vector of $\partial\Omega$ in $B^n$ with respect to the metric $g_S.$ Since $\Omega$ is strictly convex with respect to $g_S$, $\mu_n\leq 0$ (by Lemma \ref{app mun negatif}) and $g_s\geq d\varphi\otimes d\varphi$ ($\varphi$ is spacelike) we deduce that  $II_{\partial M_{\varphi},M_{\varphi}}$ is positive definite, and using (\ref{pf lem tt geod H2 II}) again that assumption (H2) holds in that case ($\langle II_{\partial M_{\varphi}},n\rangle=II_{\partial M_{\varphi},M_{\varphi}}\langle n,N_1\rangle$ with $\langle n,N_1\rangle>0$ if $n\in (T\partial M_{\varphi})^{\perp}$ is inward-directed, by definition).
\end{proof}

\section{Some algebraic properties of the curvature operators}\label{appendix properties Hm}
Let us define, for $p\in B(0,1)$ and $q'\in S_n(\R)$,
$$\mathcal{H}_m^0(p,q'):=\frac{m!(n-m)!}{n!}\frac{1}{(1-|p|^2)^{\frac{m}{2}}}F_m(A(p)^{-1}q')$$
and set, for $p\in B(0,1)\subset\R^n,$ the positive cone associated to the operator $\mathcal{H}_m^0,$
$$\Gamma_m(p)=\{q'\in S_n(\R)|\ \mathcal{H}_k^0(p,q')>0,\ k=1,\ldots,m\}.$$
The set of positivity of $\mathcal{H}_m^0$  is defined by
$$\mathcal{E}:=\{(p,q')\in B(0,1)\times S_n(\R)|\ q'\in\Gamma_m(p)\}$$
and it is well-known that the following properties hold on $\mathcal{E}$:
\begin{itemize}
\item the operator $\mathcal{H}_m^0$ is elliptic: for all $(p,q')\in\mathcal{E}$ and all $\xi\in\R^n\backslash\{0\},$
\begin{equation*}
\sum_{i,j}\frac{\partial\mathcal{H}_m^0}{\partial q'_{ij}}(p,q')\ \xi_i\ \xi_j>0\ ;
\end{equation*}
\item the operator ${\mathcal{H}_m^0}^{\frac{1}{m}}$ is concave with respect to the second variable $q'$: for all $(p,q')\in\mathcal{E}$ and all $(\xi_{ij})_{ij}\in S_n(\R)$, 
\begin{equation*}
\sum_{i,j,\ k,l}\frac{\partial^2{\mathcal{H}_m^0}^{\frac{1}{m}}}{\partial q'_{ij}\partial q'_{kl}}(p,q')\ \xi_{ij}\ \xi_{kl}\leq 0\ ;
\end{equation*}
\end{itemize}
The operator $\mathcal{H}_m(x,p,q):=\frac{1}{\mu(x)^m}\mathcal{H}_m^0(p,q')$ with $q'=q+d(\log\mu)_x(\sum_sp_se_s)A(p)$ satisfies
$$\frac{\partial\mathcal{H}_m}{\partial q_{ij}}(x,p,q)=\frac{1}{\mu(x)^m}\frac{\partial\mathcal{H}_m^0}{\partial q'_{ij}}(p,q')\hspace{.5cm}\mbox{and}\hspace{.5cm}\frac{\partial^2\mathcal{H}_m^{\frac{1}{m}}}{\partial q_{ij}\partial q_{kl}}(x,p,q)=\frac{1}{\mu(x)}\frac{\partial^2{\mathcal{H}_m^0}^{\frac{1}{m}}}{\partial q'_{ij}\partial q'_{kl}}(p,q')$$ 
so that (\ref{intro Hm elliptic Km}) and (\ref{intro Hm concave Km}) hold on admissible functions.

\section{Two lemmas on spacelike hypersurfaces in $\widetilde{AdS}^{n,1}$}\label{appendix equidistant}
The first lemma gives an elementary bound on the oscillation of a spacelike graph in $\widetilde{AdS}^{n,1}$:
\begin{lem}\label{app lem estim u}
If $u:B^n\rightarrow\R$ is a spacelike function such that $u(0)=0$ then
\begin{equation}\label{app equidistant lem estim u}
|u(x)|\leq 2\arctan(|x|)
\end{equation}
 for all $x\in B^n.$ In particular, if $|x|\leq r<1,$ $|u(x)|\leq 2\arctan(r)<\pi/2.$
\end{lem}
\begin{proof}
Let us fix $x_1\in B^n$ and consider the path $x(s)=sx_1,$ $s\in[0,1].$
Since $u(0)=0,$ we have
$$u(x_1)=\int_0^1\frac{d}{ds}(u(x(s))ds=\int_0^1g_S\left(\nabla^Su,x'\right)ds$$
and, since $|\nabla^Su|_S\leq 1,$
$$|u(x_1)|\leq\int_0^1\left|g_S\left(\nabla^Su,x'\right)\right|ds\leq \int_0^1|x'|_Sds.$$
Since $x(s)=sx_1$ we have $x'=x_1$ and $|x'|_S=\frac{2|x_1|}{1+s^2|x_1|^2},$ which yields
$$|u(x_1)|\leq \int_0^1\frac{2|x_1|}{1+s^2|x_1|^2}ds=2\arctan|x_1|.$$
\end{proof}
The second lemma shows that there exist equidistant hypersurfaces in $\widetilde{AdS}^{n,1}$ whose slope is arbitrarily close to 1 on a given compact subset. They serve as barriers for the gradient estimate on the boundary. First recall the quadric model $\mathbb{H}^{n,1}\subset\R^{n,2}$ of Anti-de Sitter geometry described in (\ref{quadric model}), and its relation to the Poincar\'e model $\widetilde{AdS}^{n,1}=B^n\times\R$ used in the paper. For $\tau_0\in (-\pi/2,\pi/2)$ the set $\{\cos\tau_0\ X+\sin\tau_0\ e_{n+2},\ X\in\mathbb{H}^{n}\}\subset\mathbb{H}^{n,1}$ is a set of points at distance $|\tau_0|$ to $\mathbb{H}^{n}$ (the distance is here the natural lorentzian distance to a given spacelike hypersurface); its equation in $B^n\times\R$ is $\mu\sin(t)=\sin\tau_0.$ Let us shift the set vertically in $B^n\times\R$, in order to obtain an equidistant hypersurface $E_{\tau_0}$ which contains the origin $(x,t)=(0,0)$; its equation is $\mu\sin(t+\tau_0)=\sin\tau_0,$ or equivalently in $\mathbb{H}^{n,1}$ 
$$\sin\tau_0\ X_{n+1}+\cos\tau_0\ X_{n+2}=\sin\tau_0.$$
We now consider the isometry $f\in Isom^0(\R^{n,2})$ given by
$$f(e_n)=\cosh\beta\ e_n+\sinh\beta\ e_{n+2},\  f(e_{n+2})=\sinh\beta\ e_n+\cosh\beta\ e_{n+2}$$
and $f(e_i)=e_i$ for $i=1,\ldots,n-1$ and $i=n+1,$ where $\beta$ is a real parameter. The equation of the equidistant $f(E_{\tau_0})$ is
$$\sin\tau_0\ X_{n+1}+\cos\tau_0\ \left(-X_n\sinh\beta+X_{n+2}\cosh\beta\right)=\sin\tau_0.$$
In the Poincar\'e model $B^n\times\R$ it reads
\begin{equation}\label{app eqn u equidistant phi} 
\sin\tau_0\cos t+\cos\tau_0\sin t\cosh\beta=\frac{2x_n}{1+|x|^2}\cos\tau_0\sinh\beta+\sin\tau_0\frac{1-|x|^2}{1+|x|^2}.
\end{equation}
It is the graph of a spacelike function $\psi:B^n\rightarrow\R$ which is such that
\begin{equation}\label{app u partial u phi}
\psi(0)=0,\ \partial_i\psi(0)=0\mbox{ for } 1\leq i\leq n-1\mbox{ and }\ \partial_n\psi(0)=2\tanh\beta;
\end{equation}
since $g_S=4|dx|^2$ at the center $0$ of $B^n,$ we obtain that $|\nabla^S\psi(0)|_S=|\tanh\beta|.$ The next lemma states that $|\nabla^S\psi|_S$ tends to 1 as $|\beta|$ tends to infinity, uniformly on compact subsets of $B^n.$
\begin{lem}\label{app norm nabla u tends 1}
For all $r\in (0,1),$ $|\nabla^S\psi|_S\rightarrow 1$ as $|\beta|\rightarrow+\infty,$ uniformly on $\overline{B}(0,r).$
\end{lem}
\begin{proof}
Writing Equation (\ref{app eqn u equidistant phi}) in the form
$$\sin t\pm \frac{2x_n}{1+|x|^2}=\frac{\tan\tau_0}{\cosh\beta}\left(\frac{1-|x|^2}{1+|x|^2}-\cos t\right)+\frac{2x_n}{1+|x|^2}(\tanh\beta\pm 1)$$
we see that $\sin t\rightarrow \pm \frac{2x_n}{1+|x|^2}$ as $\beta\rightarrow\pm\infty,$ uniformly on $B^n,$ and deduce that
\begin{equation}\label{lim cos2t uniform}
\cos^2t=1-\sin^2 t\rightarrow 1- \frac{4x_n^2}{(1+|x|^2)^2}\hspace{.5cm}\mbox{as}\hspace{.5cm}|\beta|\rightarrow +\infty,
\end{equation}
uniformly on $B^n.$ Differentiating (\ref{app eqn u equidistant phi}) we easily obtain the formula
$$|\nabla^S\psi|_S^2=\frac{1}{(1+|x|^2)^2h'^2}\left|\cos\tau_0\sinh\beta(-2x_nx+(1+|x|^2)e_n)-2\sin\tau_0\ x\right|^2$$
with $h'=-\sin\tau_0\sin t+\cos\tau_0\cos t\cosh\beta.$ We write that formula in the form
\begin{equation}\label{app equidistant nabla u A B}
|\nabla^S\psi|_S^2=\frac{A}{\cos^2t\ B}
\end{equation}
with
$$A=\left|\cos\tau_0\tanh\beta\ (-2x_nx+(1+|x|^2)e_n)-\frac{2\sin\tau_0\ x}{\cosh\beta}\right|^2$$
and
$$B=(1+|x|^2)^2\cos^2\tau_0\left(1-\frac{\tan\tau_0\tan t}{\cosh\beta}\right)^2.$$
We see that
\begin{equation}\label{app equidistant A lim uniform}
A\rightarrow \left|\cos\tau_0\ (-2x_nx+(1+|x|^2)e_n)\right|^2\hspace{.5cm}\mbox{as}\hspace{.5cm}|\beta|\rightarrow+\infty
\end{equation}
uniformly on $B^n$ and that
\begin{equation}\label{app equidistant B lim uniform}
B\rightarrow (1+|x|^2)^2\cos^2\tau_0\hspace{.5cm}\mbox{as}\hspace{.5cm}|\beta|\rightarrow+\infty
\end{equation}
uniformly on $\overline{B}(0,r)$, since $|t|\leq\pi/2-\delta$ implies that $\tan t$ is bounded. We deduce from (\ref{lim cos2t uniform})-(\ref{app equidistant B lim uniform}) that $|\nabla^S\psi|_S\rightarrow 1$ as $|\beta|\rightarrow+\infty,$ uniformly on $\overline{B}(0,r).$ In that last step, we also use that the function $A$ is bounded above and the functions $B$ and $\cos^2t$ are bounded below, since $|t|\leq\pi/2-\delta$ implies that $\cos t\geq\cos(\pi/2-\delta)>0$ and $|\tan t|\leq\tan(\pi/2-\delta).$
\end{proof}

\noindent\textbf{Acknowledgement.} The author is greatly indebted to Andrea Seppi for many enlightening discussions on Anti-de Sitter geometry and the Dirichlet problem studied in the paper. Moreover, the paper is motivated by the joint project of extending the results of \cite{BS} to Anti-de Sitter geometry.

\end{document}